\documentclass[11pt,english,a4paper]{article}

\usepackage[english]{babel}

\usepackage{amssymb}

\usepackage[T1]{fontenc}

\usepackage{enumerate}

\usepackage{amscd}

\usepackage{amsthm}

\usepackage{amsmath}

\usepackage{array}

\usepackage{delarray}

\def\s|{\, | \,}

\def\ra{\longrightarrow}

\def\emptyset{\varnothing}

\def\g{\mathbf{g}}

\def\<{\langle}

\def\>{\rangle}

\def\varps{\varepsilon}

\def\Lamb{\Lambda'}
\def\flag{\mathcal{F}}

\def\ol{\overline}

\def\L{\mathfrak{L}}

\def\N{\mathbb{N}}

\def\Z{\mathbb{Z}}

\def\Q{\mathbb{Q}}

\def\R{\mathbb{R}}

\def\C{\mathbb{C}}

\def\F{\mathbb{F}}

\def\G{\mathbb{G}}
\def\i{\iota}

\def\X{\mathbf{X}}
\def\Y{\mathbf{Y}}
\def\p{\mathfrak{p}}

\def\P{\mathbb{P}}

\def\n{\mathfrak{n}}

\def\T{\mathfrak{T}}

\def\ns{\triangleleft}

\def\GL{\mathrm{GL}}

\def\I{\mathrm{I}}

\def\isom{\cong}

\def\S{\mathcal{S}}
\def\det{\mathrm{det}^{\frac{1}{2}}}

\newtheorem{theo}{Theorem}[section]

\newtheorem{lemma}[theo]{Lemma}

\newtheorem{cor}[theo]{Corollary}

\newtheorem{prop}[theo]{Proposition}

\theoremstyle{definition}

\newtheorem{Def}[theo]{Definition}

\newtheorem{ex}[theo]{Example}

\newtheorem{remark}[theo]{Remark}

\begin{document}

\title{The geometric structure of class two nilpotent groups and subgroup growth}

\author{Pirita Paajanen\\ School of Mathematics\\Southampton
  University\\ University Road\\ Southampton SO17 1BJ, UK\\p.m.paajanen@soton.ac.uk}

\maketitle

\begin{abstract}
In this paper we derive an explicit expression for the normal zeta
function of class two nilpotent groups whose associated Pfaffian
hypersurface is smooth. In particular, 
we show how the local zeta function depends on counting
$\F_p$-rational points on related varieties, and we describe the varieties that
can appear in such a decomposition.
As a
corollary, we also establish explicit results on the degree of polynomial
subgroup growth in these groups, 
and we study the behaviour of poles of this zeta function. 
Under certain geometric conditions, we also confirm
that these functions satisfy a functional equation.
\end{abstract}

\section{Introduction}


Zeta functions of a finitely generated group $G$
were introduced in \cite{GSS} as a non-commutative
analogue to the Dedekind zeta function of a number field. They are
used to study the arithmetic and asymptotic properties of the sequence of
numbers $$a_n(G)=|\{H\leq_f G: |G:H|=n \}|.$$ The fact that $G$ is
finitely generated ensures that $a_n(G)$ is finite for all $n$.
The term {\it subgroup growth} is used to describe the study of the sequences
$a_n(G)$ and $s_n(G)=\sum_{i=1}^na_i(G)$.
 
We can also consider different variants of subgroup growth, for
example, we may only count the normal subgroups 
of finite index. To do this we define the sequence
$$a_n^\ns(G)=|\{H\ns_f G: |G:H|=n \}|.$$ Or we can define a sequence
$$a_n^\wedge(G)=|\{H\leq_f G: |G:H|=n,
\hat{H}\isom \hat{G} \}|,$$  to count those subgroups whose
profinite completion is isomorphic to the profinite completion of
the group itself.  We define the
\emph{zeta function} of $G$ to be the formal Dirichlet series
$$\zeta_G^*(s)=\sum_{n=1}^\infty
 a^*_n(G)n^{-s},$$ where $*\in\{\leq,\ns,\wedge\},$ 
and $s$ is a complex variable. 
The abscissa of convergence of $\zeta_G^*(s)$, which we denote by $\alpha_G^*$, determines \emph{the degree of
polynomial subgroup growth} since $$\alpha_G^*:=\inf\{\alpha\geq 0:\
\mathrm{there\ exists}\ c>0\ \mathrm{with}\ s_n^*(G)< cn^{\alpha}\
\mathrm{for\ all}\ n\}.$$

When $G$ is a finitely generated torsion-free 
nilpotent group, called henceforth a $\T$-group, the
function $a_n^*(G)$ grows polynomially and is multiplicative, in the sense that 
$a_n^*(G)=\prod_i a_{p_i^{k_i}}^*(G)$
where $n=\prod_i p_i^{k_i}$. Therefore, like the zeta
functions in number theory, we can write the Dirichlet series as
an Euler product decomposition of \emph{the local zeta functions}, i.e.,
$$\zeta^*_G(s)=\prod_{p\ prime} \zeta^*_{G,p}(s),$$ where
$$\zeta_{G,p}^*(s)=\sum_{k=0}^\infty a^*_{p^k}(G)p^{-ks}$$ 
counts only the subgroups with $p$-power index.

A fundamental theorem of Grunewald, Segal and Smith \cite{GSS} states 
that these local zeta functions of $\T$-groups are
rational functions in $p^{-s}$. This rationality implies
that the coefficients $a_n^*(G)$ behave
smoothly; in particular, they satisfy a linear recurrence relation. This is
proved by expressing the zeta function as a $p$-adic integral and then
using a model theoretic black box to deduce the rationality.
Results of Denef on $p$-adic integrals arising 
from a local Igusa zeta function have been generalised  
by Grunewald and du Sautoy in \cite{Annals} and by Voll \cite{Vollnew}.
This has lead to
some new interesting problems. For example it transpires that 
the Hasse-Weil zeta function
of smooth projective varieties is possibly a 
better analogue for the zeta functions in the group setting than the Dedekind
zeta function of number fields, as previously mentioned.
This initiated an arithmetic-geometric approach 
to zeta functions of groups. Theorem 1.6 in
\cite{Annals} gives an explicit decomposition for the zeta function as
a sum of rational functions $P_i(p,p^{-s}),$  with coefficients
coming from counting points mod $p$ on various varieties. 
In a subsequent paper \cite{Elliptic} du 
Sautoy presents a group whose zeta function depends
on the $\F_p$-rational points on an elliptic curve.
The explicit zeta function $\zeta_{G,p}^\ns(s)$ of 
this elliptic curve example is calculated 
by Voll in \cite{Amer}, and he also proves that it satisfies a
functional equation.
Since the zeta function depends on counting the number of
$\F_p$-points on the elliptic curve, an expression that is 
more complicated than just a polynomial in $p$, the 
proof that it satisfies a functional equation uses 
the fact that the Hasse-Weil zeta function of smooth 
projective varieties also satisfies a functional equation.
In \cite{FE}, he extends this to all smooth hypersurfaces, 
with the restriction that these hypersurfaces should not contain any lines.

In the current paper, we generalise Voll's 
work by removing the condition on linear subspaces.
We still require that the hypersurface associated with the groups is
smooth.
The main contribution of this generalisation is that it sheds some light on which
varieties can arise in the context of zeta functions of groups. 
In particular, we can describe the varieties  and the
numerical data associated to the poles of the zeta function quite
explicitly. 

Let us now define the particular varieties that will be of interest to us.

\begin{Def}\label{pfaffian}
Let $\Gamma$ be a class 2 $\T$-group with $Z(\Gamma)=[\Gamma,\Gamma]$, and a presentation of the form
$$\Gamma:=\<x_1,\dots,x_d,y_1,\dots,y_{d'}:[x_i,x_j]=M(\mathbf{y})_{ij} \>,$$ where
$d$ is even and
$M(\mathbf{y})$ is the matrix of relations; an antisymmetric $d\times d$ matrix with entries that are
linear forms in the $y_i$. We call the variety defined by the equation
\begin{equation*}
\mathfrak{P}_\Gamma:   \det M(\mathbf{y}) =0
\end{equation*}
the {\it Pfaffian hypersurface associated to $\Gamma$.} We exclude from
this definition the case when the determinant is identically zero. 

We shall extend this definition for a general 
class 2 $\T$-group $G$ by writing $G=\Gamma\times \Z^m$, where
$\Gamma$ is as above and $m\geq 0$. Then $\mathfrak{P}_G:=\mathfrak{P}_\Gamma$.
\end{Def}

Note that the condition that $d$ is an even number is necessary to
ensure that the corresponding Pfaffian hypersurface is
non-trivial. Similarly the extended definition, we want to exclude the
$\Z^m$, because otherwise the Pfaffian would
be identically zero.

\begin{Def}\label{grassmannian} The set of  $(k-1)$-planes in $\P^{d'-1}$ forms a variety called the Grassmannian, denoted by $\G(k-1,d'-1).$  The affine analogue is the set of $k$-planes in affine $d'$-dimensional space, which we denote by $G(k,d')$.
\end{Def}

It is well-known that $\G(k-1,d'-1)$ can be embedded
in $\P^N$ via the Pl\"ucker embedding, where $N = \binom{n}{k}$. 

\begin{Def}\label{fano}
The Fano variety associated to a hypersurface $X\subset \P^{d'-1}$ is the variety 
$$F_{k-1}(X)=\{\Pi \in\G(k-1,d'-1): \Pi\subset X\} $$ 
of $(k-1)$-planes contained in $X$.
\end{Def}

By abuse of notation we will identify $(k-1)$-dimensional linear spaces
$\Pi \in \P^{d'-1}$ with the corresponding point $\Pi\in \G(k-1,d'-1)$.

We shall show in this paper that we can derive a decomposition of the
normal local zeta function of $G$, which is similar to that of Theorem
1.6 in \cite{Annals}, and also describe
the varieties that appear in our decomposition exactly as the
Pfaffian hypersurface, and components 
of the Fano varieties of the linear subspaces contained
on the Pfaffian. 
Here  components may arise because the Fano variety 
may be reducible and, moreover, since we consider
determinantal varieties, the rank of the defining matrix needs to
be taken into account, and different components may give different ranks.

The local zeta functions of finitely generated torsion-free nilpotent
groups calculated over the years, see  \cite{zetabook} for a
collection of these, exhibit curious symmetries
which suggest the possibility all zeta functions of groups
may have a similar local functional equation. 
In a recent paper \cite{Vollnew}, 
Voll has shown that a functional equation is satisfied for local zeta
functions of $\T$-groups counting all subgroups, conjugacy classes of subgroups and
representations. 
However, for the local normal zeta function Voll's proof works only 
in nilpotency class two. Explicitly, let $G=\Gamma\times \Z^m$ be a class two
nilpotent group, where $Z(\Gamma)=[\Gamma,\Gamma]$ with
$h(\Gamma^{ab})=d$ and $h(Z(\Gamma))=d'$, which yields $h(G)=d+d'+m$,
$h(Z(G))=m+d'$, $h([G,G])=d'$, $h(G/[G,G])=m+d$ and $h(G/Z(G))=d$. Here  $h$ is the Hirsch length,
which for a finitely generated group $G$ is 
the number of
infinite cyclic factors in a subnormal series of $G$.
Then, as Voll shows,
the functional equation in the class two is of
the form \begin{equation*}
\zeta_{G,p}^\ns(s)|_{p\mapsto p^{-1}} =
(-1)^{d+d'+m}p^{\binom{d+d'+m}{2}-(2d+d'+m)s} \zeta_{G,p}^\ns(s).
\end{equation*}

Examples of local normal zeta functions of  groups of nilpotency class three
calculated by Woodward \cite{zetabook} show that a
functional equation is not always satisfied. However, \cite{zetabook}
contains also a number of examples of local normal zeta functions in
higher classes where the functional equation holds. It remains
mysterious when there is a functional equation for the normal zeta function
and when not.

Other type of results on functional equations connected to zeta
functions of groups include work by Lubotzky and du Sautoy
\cite{algebraic} on zeta functions counting those subgroups whose
profinite completion is isomorphic to the profinite completion.
These results have been partially generalised by
Berman \cite{Mark}.

In the current paper we confirm with a different method than Voll's 
that the functional equation is satisfied in the case
where the Pfaffian hypersurface 
associated to the group, and the components of its Fano varieties, are smooth
and absolutely
irreducible.
This is the best possible result using our
methods, since the final step uses the fact that the Hasse-Weil
zeta function of points on a smooth 
and absolutely irreducible algebraic variety satisfies a
functional equation.

Our main theorem also has applications to subgroup growth. In particular, we provide evidence that some of the best bounds known
so far for subgroup growth may be exact. 
We also deduce results on the number and nature of poles of the normal zeta function for class two nilpotent groups.


We refer the reader to \cite{ICM} for the 
most resent survey of the general theory of these zeta functions.

\subsection{Results}

Before stating our main results, we begin with a few preliminary definitions.

A classical result, e.g. \cite{GSS} states that
the zeta function of the free abelian group of rank $d$ 
can be expressed in terms of $\zeta(s)$, the Riemann zeta function, as follows:
$$\zeta_{\Z^d}(s)=\zeta(s)\zeta(s-1)\dots\zeta(s-(d-1)).$$  
We write $\zeta_p(s)=\frac{1}{1-p^{-s}}$ for its local factors, where $p$ is a prime.

\begin{Def} 
A \emph{flag} of type $I$ in $V=\F_p^{d'}$, where
$I=\{i_1,\dots,i_l\}_<\subseteq\{1,\dots,d'\},$ (where 
$\{i_1,\dots,i_l\}_<$  is an indexing set ordered such that $i_1<i_2<\dots<i_l$) is a sequence $(V_{i_j})_{i_j\in I}$
of incident vector spaces
$$\{0\}<V_{i_1}<\dots<V_{i_l}\leq V$$ with
$\mathrm{dim}_{\F_p}(V_{i_j})=i_j$. The flags of type $I$ form a projective variety
$\flag_I$, whose number of $\F_p$-points is denoted by
$b_I(p)\in\Z[p]$, where $$b_I(p)=\binom{d'}{i_l}\binom{i_l}{i_{l-1}}\dots\binom{i_3}{i_2}\binom{i_2}{i_1}.$$
\end{Def}

\begin{Def}\label{I-factor}
The \emph{Igusa factor}
is a rational function in the set of variables $\mathbf{U}=\{U_1,\dots,U_n\}$ defined by $$I_n(\mathbf{U})=\sum_{I\subseteq \{1,\dots,n\}}
b_I(p^{-1})\prod_{i\in I}\frac{U_i}{1-U_i},$$ where $b_I(p)$ is
the number of flags of type $I$ in  $\F_p^{n+1}$.
\end{Def}

\begin{Def}
We call a rational function of the form
\begin{equation*} E_{\iota_i}(X_i,Y_{\iota_i}) =
\frac{p^{-d_{\iota_i}}Y_{\i_i} - p^{-n_i}X_i}{(1-X_i)(1-Y_{\i_i})}
\end{equation*} the $\i_i^{th}$ \emph{exceptional factor} with numerical
data $n_i$ and $d_{\i_i}$.
\end{Def}

Our main theorem is the following
\begin{theo}\label{main theorem general}
Let $G=\Gamma\times \Z^m$ be a class two nilpotent group such that
$Z(\Gamma)=
[\Gamma,\Gamma]$ and 
$d=h(G/Z(G))$ is an even number. Set $d'=h([G,G])$  
and let $\mathfrak{P}_G$ be the
associated Pfaffian hypersurface, which we assume to be reduced, smooth and non-zero.
Then for almost all primes $p$
\begin{equation*} \zeta_{G,p}^\ns(s)=
\zeta_{\Z^{d+m},p}(s)\zeta_p((d+d')s-d(d+ m))(W_0(\X,\Y) +
\sum_{i=1}^k\sum_{\i_i} \n_{\i_i}(p)\delta_{\i_i} W_{\i_i}(\X,\Y)),\end{equation*}
where the inner sum is over components 
$F_{\i_i}$ of $F_{i-1}( \mathfrak{P}_G)$, and $\delta_{\i_i}$ is the
Kronecker delta function. The 
$\n_{\i_i}(p)$ denotes
the number of $\F_p$-rational points on
these components, and
$$W_{\i_i}(\X,\Y)=\I_{d'-i-1}(X_{d'-1},\dots,X_{i+1})E_{\i_i}(X_i,Y_{\i_i})
\I_{i-1}(Y_{\i_{i-1}},\dots,Y_{\i_1}).$$
Here $\I_n$ is an Igusa
factor with $n$ variables and we set $E_0=I_0=I_{-1}=0.$
The numerical data is
\begin{align*}X_i&=p^{i(d+d'+m-i)-(d+i)s}, i\geq 1\\
Y_1&=p^{(d+m+d_1)-(d-1)s},\\
Y_{\i_i}&=p^{(i(d+m)+d_{\i_i} )-(d+i-1)s}, i\geq 2.
\end{align*}
where $d_{\i_i}$ is the dimension of $F_{\i_i}$.
\end{theo}

\begin{remark} When $\delta_{\i_i}$ is 1 or 0
 is a technical condition, that depends
 on the possible reducibility of the Fano variety and the rank of the
 matrix defining the Pfaffian hypersurface. We will delay the exact
 definition to Remark \ref{corank 0}. 
\end{remark} 


\begin{remark}
Lemma \ref{bound for poles} below gives a generic bound $k\leq
5$. This gives us also a good bound for the number of poles as in Corollary \ref{no of poles}.
\end{remark}

In all the following Corollaries, we will assume $G$ satisfies the conditions of the main theorem. Moreover, in the results about 
the degree of polynomial
subgroup growth, we have to assume that the  Pfaffian hypersurface 
associated to $G$ 
behaves as a generic hypersurface of degree $\frac{d}{2}.$ This 
will be explained in Section \ref{absofcon}.

\begin{cor}\label{real abscissa}
The abscissa of convergence is \begin{equation}\label{abs1}
\alpha_G^\ns=\max_{1\leq i\leq d'-1}\left\{d+m,
\frac{i(d+d'+m-i)+1}{(d+i)} \right\}.\end{equation}
\end{cor}

Note that this is the lower limit given in
\cite[Theorem 1.3]{Abscissa}, except that we have removed the
restriction on $Z(G)=[G,G]$.

The next corollary follows immediately from Corollary \ref{real abscissa}.
\begin{cor}\label{small centre}
Let $G$ be a class two nilpotent group with $d'=h([G,G])\leq 6$, such that the Pfaffian hypersurface associated to $G$ is a smooth generic hypersurface. Then
$\alpha_G^\ns=d+m$.
\end{cor}

The proof of Corollary \ref{small centre} uses the fact that a generic hypersurface is smooth in
$\P^{d'-1}$, if $d'\leq 6$ (see Remark 8.3 in \cite{Beauville}), and that the maximum in (\ref{abs1}) is attained at $d+m$ if $4d(d+m)>d'^2.$

In number theory there are much  finer asymptotic estimates
concerning the asymptotic behaviour of the coefficients of a
Dirichlet series; these are the so-called Tauberian theorems. To apply these results we need to be able to analytically continue the
zeta function to the left of its abscissa of convergence. The
analytic continuation by some $\varepsilon >0$ for zeta functions of
$\T$-groups is proved in \cite{Annals} using the analytic continuation
of the Artin $L$-functions.

Using the combination of Theorem \ref{main theorem general} and the appropriate Tauberian theorems it is possible to obtain  more detailed results on subgroup
growth.

\begin{cor}
Let $\beta_G^\ns\in \N$ be the multiplicity of the pole of the zeta function located at
the abscissa of convergence. Then exists $c_G\in\R$ such that
\begin{equation*} s_n^\ns(G) \sim c_G
\cdot n^{\alpha_G^\ns}(\log n)^{\beta_G^\ns-1}.
\end{equation*} as $n\rightarrow \infty$.
\end{cor}

\begin{remark}
It is clear from the statement of Theorem \ref{main theorem general}
that all the poles are 
simple, except when the numerical data is such that it
produces a multiple pole. In particular, we have
$\beta_G^\ns=1$.
\end{remark}

Theorem \ref{main theorem general} indicates that the existence of a multiple pole is
just a coincidence of
the numerical data, rather than any structural properties of the
function or the group. It should be noted that these coincidences exist for smooth Pfaffians. Moreover,
it is not difficult to prove that certain types of singularities on
the Pfaffian always produce genuine multiple poles, for instance the
ordinary double point at the origin, as is the case in the Pfaffian hypersurface for the group
$$U_3:=\<x_1,\dots,x_4,y_1,y_2,y_3:[x_1,x_2]=y_1,[x_2,x_3]=y_2,
[x_3,x_4]=y_3\>,$$ produces a double pole in the zeta function.

Our final corollary bounds the number of poles for all these zeta functions.   This follows immediately from Theorem
\ref{main theorem general}.

\begin{cor}\label{no of poles}
Let $G$ be a group as in Theorem
\ref{main theorem general}, and assume that the Pfaffian 
is a generic hypersurface of degree $\frac{d}{2}$. Then the number of 
poles of the local zeta function $\zeta_{G,p}^\ns(s)$  is at most
$d+d'+m+r$ 
where $r$ is the number of irreducible components of the 
Fano varieties on the Pfaffian hypersurface.
\end{cor}

Finally, we observe the functional equations, proved by Voll \cite{Vollnew}.
\begin{cor}\label{fn}
Assume the Pfaffian hypersurface of $G$  is absolutely irreducible, and
that the components of the  Fano varieties appearing in the 
decomposition of the zeta function of $G$ are smooth and absolutely irreducible.
Then the local normal zeta function satisfies a functional
equation of the form
\begin{equation*}\zeta_{G,p}^\ns(s)|_{p\mapsto
p^{-1}} = (-1)^{d+d'+m}p^{\binom{d+d'+m}{2}-(2d+d'+m)s} \cdot
\zeta_{G,p}^\ns(s).  \end{equation*}
\end{cor}

It would be interesting to know if all the components of the 
Fano varieties corresponding to a smooth irreducible Pfaffian 
have the same dimension. For quadrics this is known, also for 
the Fano variety of lines on a cubic hypersurface (see
\cite{Altman}). 
However, for singular Pfaffians this is not true. In \cite{H-B} 
Browning and Heath-Brown give an example of a cubic hypersurface
defined by the equation $y_1y_2y_3+y_1^2y_4+y_2^2y_5=0$, whose Fano
variety of planes consists of $\P^1$ and a point. 
This hypersurface can be encoded as the Pfaffian of the group
\begin{eqnarray*}G =& <x_1,\dots,x_6,y_1,\dots,y_5: [x_1,x_4 ]=y_1,
[x_1,x_5 ]=y_2,[x_2,x_4 ]=y_5,\\ & [x_2,x_5 ]=y_3,[x_2,x_6 ]=y_4,[x_3,x_5]=y_1,[x_3,x_6 ]=y_2>.\end{eqnarray*}
However, this Pfaffian hypersurface is not smooth and hence our 
theorem does not apply to this group. The splitting of the Fano 
variety of planes in this case is curious and deserves some further 
investigation. 

\subsection{Layout of the paper}

Our results on subgroup growth and abscissae of convergence, 
namely Corollaries \ref{real abscissa}-\ref{no of poles}, are proved in Section \ref{absofcon}.
In Section \ref{functeqn} we briefly comment on how the functional
equation described in Corollary \ref{fn} can be deduced from the main theorem. 

Next in Section \ref{correspondence} we explain how to enumerate
lattices in $\Z_p^m$, and calculate the number of lattices of a fixed
elementary divisor type.
In Section \ref{flags} we define Grassmannian and flag varieties, 
give coordinates for the Pl\"ucker embedding of the
Grassmannian 
and define flags in terms of these coordinates. We also determine when two different  lattices lift the same flag.
In Section \ref{congruences} we shall discuss solution 
sets of systems of linear  congruences, define Pfaffian hypersurfaces
and give a careful description of the geometry needed to find a
solution set of general congruences.
In Section \ref{valuations} we compute the $p$-adic valuations which arise in the solution sets.

In the Sections \ref{basic theory} to \ref{Igusa} we decompose 
the zeta function and apply our earlier work to complete the proof of Theorem \ref{main theorem general}. Section \ref{segre} is an explicit example and illustration of the general theory.

\section{Abscissae of convergence}\label{absofcon}

In this section we prove the Corollary \ref{real abscissa} on the rate of subgroup growth and prove a result about the number of poles of the zeta function in the case that the Pfaffian hypersurface is generic; generic meaning that, there is a non-empty open subset $U$ of the parameter space of hypersurfaces such that for every point in $U$ the answer to the property we are asking, is uniform.

Our main theorem, Theorem \ref{main theorem general}, is stated in terms of local zeta functions. However, the degree of polynomial subgroup growth is equal to
the abscissa of convergence of the global zeta function. Thus we need to analyse
the convergence of an infinite product. This can be done using the following well-known results.

(A) An infinite product $\prod_{n\in J} (1+a_n)$ converges
absolutely if and only if the corresponding sum $\sum_{n\in J}
|a_n|$ converges.

(B) $\sum_{p} |p^{-s}|$ converges at $s\in\C$ if and only if
$\Re(s)>1$.

In view of (B),  $\frac{1}{1-p^{-a_is+b_i}}$ has abscissa of converge $\frac{b_i}{a_i}$, so (A) implies that $\prod_{p\ prime} \frac{1}{1-p^{-a_is+b_i}}$ has abscissa of convergence $\frac{b_i+1}{a_i}$.

If we can show that
the abscissa of
convergence is determined by the denominator of the zeta function, then we can use the above observation to deduce the desired result, namely Corollary \ref{real abscissa}.

First, we note that the Igusa-type factors in the numerator do not affect the abscissa of
convergence. Indeed, they are a combination of the numerical data $p^{-a_is+b_i}$ from the denominator, multiplied
with sums of powers of $p^{-1}$ which  come from the factors $b_I(p^{-1})$. In particular, all the Igusa type of factors in the numerator will converge  no worse than
the abscissa of convergence coming from the denominator.

The second main ingredient in the explicit formula in the statement of the Theorem \ref{main theorem general} are the $\i_i^{th}$ \emph{exceptional factors}
in variables $X_i$ and $Y_{\i_i}$, which are of the form
\begin{equation*} E_{\i_i}(X_i,Y_{\i_i}) =
\frac{p^{-d_{\i_i}}Y_{\i_i} - p^{-n_i}X_i}{(1-X_i)(1-Y_{\i_i})}.
\end{equation*}
To analyse this factor we need to take into account the numbers $\n_{\i_i}(p).$
By the Lang-Weil estimates  \cite{Lang-Weil}, we have $\n_{\i_i}(p)= \delta p^{d_{\i_i}} + O(p^{d_{\i_i}-\frac{1}{2}}),$ where $\delta=(\frac{d}{2}-1)(\frac{d}{2}-2).$
We can rewrite the numerator of the exceptional factor as
$p^{-d_{\i_i}}Y_{\i_i} - p^{-n_i}X_i= p^{-d_{\i_i}}(Y_{\i_i}-p^{-c_{\i_i}}X_i),$
so the numerator of $\n_{\i_i}(p)E_{\i_i}(X_i,Y_{\i_i})$ is equal to $\delta Y_{\i_i}-\delta p^{-c_{\i_i}}X_i+
O(p^{-\frac{1}{2}}Y_{\i_i})+O(p^{-c_{\i_i}-\frac{1}{2}}X_i)$ and this will converge no worse than the denominator $(1-X_i)(1-Y_{\i_i})$ for  large enough $p$.

This gives us a first estimate for the abscissa of convergence just by recording the numerical data from the poles, namely
\begin{equation}\label{abscissa}\alpha_G^\ns=\max_{2\leq i\leq d'-1}
\left\{d+m,\frac{i(d+d'+m-1)+1}{d+i},\frac{i(d+m)+d_{\i_i}+1}{d+i-1},
\frac{d+m+d_1+1}{d-1},\frac{d+d'+m}{d+1}
\right\}.
\end{equation}

In order to finish 
the proof of the Corollary \ref{real abscissa} we need to estimate the $d_{\i_i}$
appearing in the numerical data. 

We can make some preliminary observations and reductions. 
In particular, it suffices to consider only the cases where $d\geq
6$. 
Indeed, if  $d=2$ then we may assume $G$ is the  Heisenberg 
group and this is known to have a normal zeta function with
abscissa of convergence 2, see \cite{GSS}.
In the case $d=4$, the Pfaffian hypersurface is a quadric, 
and the smooth projective quadrics over $\F_p$
have been classified, see e.g. \cite{Hirschfeld}. 
By inspecting an explicit list we can use Theorem \ref{main theorem
  general} 
to deduce that the abscissa of convergence is always 4.
We also note that $d'\leq \frac{d(d-1)}{2}$, with equality 
for the free class two nilpotent groups.
Trivially we have $1\leq i\leq d'.$

If we assume that the Pfaffian hypersurface of degree $\frac{d}{2}$
behaves like a generic hypersurface
in $\P^{d'-1}$, then the
dimension of the Fano variety of $(i-1)$-planes is
\begin{equation}\label{dim fano}d_{\i_i}=i(d'-i)-\binom{\frac{d}{2}+({i-1})}{i-1}, \end{equation}for $\frac{d}{2}\geq 3$ (see e.g. \cite[Theorem 12.8]{Harris}).
Note that if $d_{\i_i}<0$, then the Fano variety of $(i-1)$-planes on the Pfaffian is empty.

\begin{lemma}\label{bound for poles}
Let $\mathfrak{P}_G$ be a smooth and generic Pfaffian hypersurface 
associated to a group $G$ as in Theorem \ref{main theorem general}. Then $F_4(\mathfrak{P}_G)$ is the highest non-empty Fano variety.
\end{lemma}

\begin{proof}
We need to determine for which $d,d',i$ we have
$$d_{\i_i}=i(d'-i)-\binom{\frac{d}{2}+({i-1})}{i-1}\geq 0,$$  under the additional hypothesis $d\geq 6$, $d'\leq \frac{d(d-1)}{2}$
and $1\leq i\leq d'$. These additional conditions are
all trivial bounds coming from the group structure that governs the Pfaffian hypersurface.  
In general the binomial term grows in $d^{i-1}$, while  $d'$ grows at most in $d^2$, so asymptotically the binomial term surpasses
the quadratic term
quickly. In order to make this precise for small $d$ and $d'$.
We first observe that 
\begin{equation*}i(d'-i)-\binom{\frac{d}{2}+({i-1})}{i-1}\geq 0\end{equation*}
if and only if
\begin{equation*}
d' \geq \frac{1}{i}\binom{\frac{d}{2}+({i-1})}{i-1}+i.\end{equation*}
Trivially, we have $d'\leq \frac{d(d-1)}{2}$ so it is enough to show that
$$\frac{d(d-1)}{2}\geq\frac{1}{i}\binom{\frac{d}{2}+({i-1})}{i-1}+i$$
fails
for all  large enough $d$ and $i$. It is easy to calculate that if $i=6$ then
this inequality holds only if $d\leq 5$, which means that $d\leq 4$ and we have already dealt with this case. Thus it is enough to consider
$i\leq 5,$ as claimed
\end{proof}

For hypersurfaces that are assumed to be smooth, but not necessarily
generic  a proposition of Starr from the Appendix of \cite{H-B}
states that $F_{i-1}(\mathfrak{P}_G)=\emptyset$ for
$i>\frac{d'}{2}$. Beheshti \cite{Behe} has given some bounds for these
dimensions in characteristic zero. However, we do not know enough
about these dimensions in characteristic $p$ and so we only consider
generic hypersurfaces.

We have from (\ref{abscissa}) $$\alpha_G^\ns=\max_{2\leq i\leq d'-1}\left\{d+m,
\frac{i(d+d'+m-i)+1}{(d+i)},\frac{i(d+m)+d_{\i_i}+1}{d+i-1},
\frac{d+m+d_1+1}{d-1} ,\frac{d+d'+m}{d+1} \right\}.$$ 

We shall prove the corollary in two stages, first when $4d(d+m) > d'^2$ and then
$4d(d+m) \leq d'^2$. In the first case $4d(d+m)>d'^2$ and generalising the results from
\cite{Abscissa} that
$$\max_{1\leq i\leq d'-1}\left\{d+m,\frac{i(d+d'+m-i)+1}{d+i}
\right\}=d+m.$$ 
It is an easy case analysis to show that
$$\alpha_G^\ns=\max_{2\leq i\leq 5 }\left\{d+m,\frac{i(d+m)+d_{\i_i}+1}{d+i-1},
\frac{d+d'+m}{d+1} \right\}=d+m,$$ 
using the estimate for the dimension and the inequality
 $4d(d+m) > d'^2$, and we leave this to the reader.

In the second case, $4d(d+m)\leq d'^2$. We again leave $i=1$ to the reader, 
and we have from (\ref{abscissa}) 
$$\alpha_G=\max_{2\leq i\leq d'-1}\left\{
\frac{i(d+d'+m-i)+1}{(d+i)},\frac{i(d+m)+d_i+1}{d+i-1} \right\}.$$ 
To prove the corollary, by the Lemma \ref{bound for poles}, 
we need to show that
$$\max_{2\leq i\leq 5}\left\{\frac{i(d+m)+d_i+1}{d+i-1} \right\}< 
\max_{1\leq i\leq d'-1}\left\{\frac{i(d+d'+m-i)+1}{(d+i)}\right\}.$$
Let us take $i=\frac{d'}{2}$ on the right hand side of the above equation. Then
$$\max_{1\leq i\leq d'-1}\left\{\frac{i(d+d'+m-i)+1}{(d+i)}\right\}\geq
\frac{\left
(\frac{d'}{2}\right)\left(d+d'+m-\frac{d'}{2}\right)+1}
{d+\frac{d'}{2}}= \frac{\frac{d'}{2}\left(d+m+\frac{d'}{2}
  \right)+1}{d+\frac{d'}{2}}\geq \frac{d'}{2}$$ and we have a
simple bound
 for the quantity in question from below.
Using this bound, a case analysis for any $i=1,\dots, 5,$ shows that
\begin{equation*}
\frac{i(d+m)+d_i+1}{d+i-1} \leq \frac{d'}{2}\end{equation*}
since
$d\geq 6$.
This finishes the case $4d(d+m)\leq d'^2$ and hence the proof.

\section{The functional equation}\label{functeqn}

In this section we comment on the Corollary \ref{fn}, namely that, 
the local normal zeta function of a group $G$ as in Theorem \ref{main theorem general}, 
satisfies a functional equation of the form
\begin{equation*}\zeta_{G,p}^\ns(s)|_{p\mapsto
p^{-1}} = (-1)^{d+d'+m}p^{\binom{d+d'+m}{2}-(2d+d'+m)s} \cdot
\zeta_{G,p}^\ns(s).  \end{equation*}
By the Observation 2 on page 1013 \cite{Amer}, in our formulation, this is equivalent to
\begin{equation*}
A(p,p^{-s})|_{p\mapsto
p^{-1}}=(-1)^{d'-1}p^{\binom{d'}{2}}\cdot A(p,p^{-s}),\end{equation*}
where
$$A(p,p^{-s})=
W_0(p,p^{-s}) +
\sum_{i=1}^k\sum_{\i_i} \n_{\i_i}(p) W_{\i_i}(p, p^{-s}).$$ Recall from Theorem \ref{main theorem general} that $$W_{\i_i}(\X,\Y)=\I_{d'-i-1}(X_{d'-1},\dots,X_{i+1})E_{\i_i}(X_i,Y_{\i_i})
\I_{i-1}(Y_{\i_{i-1}},\dots,Y_{\i_1}).$$

To prove the Corollary \ref{fn} it is enough to show
that each of the summands $\n_{\i_i}(p)W_{\i_i}(\X,\Y)$, for $i\geq 0$, satisfy the same functional equation
as $A(p,p^{-s})$. We shall first establish inversion properties for the $W_{\i_i}(\X,\Y)$, and secondly define what the formal inversion of $p\mapsto p^{-1}$ means for the coefficients $\n_{\i_i}(p)$.

Following Igusa \cite{Igusa}, it is proved in \cite[Theorem 4]{FE} 
that for $\mathbf{U}=(U_1,\dots,U_n)=(p^{-a_1s+b_1},\dots,p^{-a_ns+b_n})$ we have
$$I_n(\mathbf{U})|_{U_i\mapsto U_i^{-1}} = (-1)^{n}
p^{\binom{n+1}{2}} I_n(\mathbf{U}),$$ where $I_n(\mathbf{U})$ as in Definition \ref{I-factor}.
Here the map $U_i\mapsto U_i^{-1}$ corresponds to $p\mapsto p^{-1}$.

So for each
$$W_{\i_i}(\X,\Y)=I_{d'-i-1}(X_{d'-1},\dots,X_{i+1})
E_{\i_i}(X_i,Y_{\i_i})I_{i-1}
(Y_{\i_{i-1}},
\dots,Y_{\i_1})$$
we get functional equation coefficients  $(-1)^{d'-i-1} p^{\binom{d'-i}{2}}$ and $(-1)^{i-1}
p^{\binom{i}{2}}$ from the Igusa factors. It remains to determine the functional
equation for $E_{\i_i}$. This is now easy, since we have the explicit
formula $$E_{\i_i}(X_i,Y_{\i_i})=
\frac{p^{-d_{\i_i}}Y_{\i_i}-p^{-n_i}X_i}{(1-X_i)(1-Y_{\i_i})},$$ which yields
$$E_{\i_i}(X_i,Y_{\i_i})|_{p\mapsto p^{-1}} = p^{n_i+d_{\i_i}} E_{\i_i}(X_i,Y_{\i_i}).$$
Here $n_i=i(d'-i)$ is the dimension of the Grassmannian $\G(n-1,d'-1)$,
and $d_{\i_i}$ is the dimension of the component $F_{\i_i}$ on $F_{i-1}(\mathfrak{P}_G)$.

Recall, that $\n_{\i_i}(p)$ denoted the number of $\F_p$-points on certain varieties. We shall formally define $\n_{\i_i}(p)_{p\mapsto p^{-1}}=p^{-d_{\i_i}}\n_{\i_i}(p)$.
This can be thought of coming from the fact
that the Hasse-Weil zeta function has a functional equation
if the variety is smooth and absolutely irreducible.
In that setting, together the rationality and the Riemann hypothesis for the Hasse-Weil zeta function 
imply that
if the components of the Fano varieties are smooth and absolutely irreducible then $$\n_{\i_i}(p)=p^{d_{\i_i}}+\dots+1+(-1)^{m}\sum_j \pi_j,$$
where $\pi_j\in\C$ and $|\pi_j|=\sqrt{p}$. Further,
the functional equation of the zeta function implies that $\pi_j\mapsto \frac{p^{d_{\i_i}-1}}{\pi_j}$ induces a permutation of the set $\{\pi_j\}$.
It follows that the inversion 
$\n_{\i_i}(p)|_{p\mapsto p^{-1}}=
p^{-d_{\i_i}} \n_{\i_i}(p)$ is well-defined. This yields
$$W_{\i_i}(\X,\Y)|_{p\mapsto p^{-1}}=
(-1)^{d'-i-1+(i-1)}p^{\binom{d'-i}{2}+n_i+\binom{i}{2}} W_{\i_i}(\X,\Y).$$
Since,
$\binom{d'-i}{2}+n_i+\binom{i}{2}= \binom{d'}{2}$, we get $$W_{\i_i}(\X,\Y)|_{p\mapsto p^{-1}}= (-1)^{d'}p^{\binom{d'}{2}}
W_{\i_i}(\X,\Y),$$ as required.

A paper by Debarre and Manivel \cite{Manivel} shows that if $\mathfrak{P}_G$ is a generic hypersurface, and moreover a complete intersection over an algebraically closed field, then the Fano variety is smooth, connected and has the expected dimension, i.e., \begin{equation*}d_{\i_i}=i(d'-i)-\binom{\frac{d}{2}+({i-1})}{i-1}. \end{equation*}

\section{Lattices and points in the projective space}\label{correspondence}

A lattice $\Lamb$ is an additive subgroup of $\Z_p^{d'}$. We say that
$\Lamb$ is maximal in its homothety class if  $\Lambda'\leq \Z_p^{d'}$ but $p^{-1}\Lamb\not\leq\Z_p^{d'}$.
The lattices which are maximal in their class are enumerated by elementary
divisor types.

The \emph{type} of $$\Lamb\isom diag(\underset{i_1}{\underbrace{
    p^{r_{i_1}+\dots+r_{i_l}},\dots, p^{r_{i_1}+\dots+r_{i_l}}}},
\underset{i_{2}-i_{1}}{\underbrace{p^{r_{i_2}+\dots+r_{i_l}},\dots,p^{r_{i_2}+\dots+r_{i_l}}}},
\dots,\underset{d'-i_l}{\underbrace{1,\dots,1 }})$$ is denoted by $\nu=(I,r_I),$ where
$I=\{i_1,\dots,i_l\}_<\subseteq\{1,\dots,d'-1\},$ with
$i_1<i_2<\dots<i_l,$ and the vector $r_I=(r_{i_1},\dots,r_{i_l})$ records the values of the $r_{i_j}.$
If we are only interested in the indices appearing in the type, and not the exact
values of the $r_{i_j}$, we say that $\Lamb$ has  {\it flag type} $I$.

Let $\Lamb$ be a maximal lattice of type $\nu=(I,r_I),$ as above.
The group $\GL_{d'}(\Z_p)$ acts transitively on the set of maximal
lattices, and the Orbit-Stabiliser Theorem gives a 1-1
correspondence between \begin{equation}\label{corr}\{\text{maximal lattices of type}\
\nu\}\overset{1-1}{\longleftrightarrow} \GL_{d'}(\Z_p)/G_\nu,\end{equation} where
$G_\nu$ is the  stabiliser of the diagonal matrix
$$diag(\underset{i_1}{\underbrace{ p^{r_{i_1}+\dots+r_{i_l}},\dots,
    p^{r_{i_1}+\dots+r_{i_l}}}},\underset{i_{2}-i_{1}}
{\underbrace{p^{r_{i_2}+\dots+r_{i_l}},\dots,p^{r_{i_2}+\dots+r_{i_l}}}},
\dots,\underset{d'-i_l}{\underbrace{1,\dots,1 }})$$
in $\GL_{d'}(\Z_p)$.
 
The matrices in the  stabiliser $G_\nu$ are of the form
\begin{displaymath}
\begin{array}({c| c  |  c  | c |  c })
\GL_{i_1}(\Z_p)& 
 *& * & \dots& 
*  \\ 
 \hline   p^{r_{i_1}}\Z_p & \GL_{i_2-i_{1}}(\Z_p) &   *
&\dots &* \\ \hline p^{r_{i_{1}}+r_{i_2}}\Z_p& p^{r_{i_{2}}}\Z_p &  \GL_{i_{3}-i_{2}}(\Z_p)  & \dots  &* 
\\ \hline
\vdots & \vdots & \vdots &\ddots &\vdots\\ \hline
p^{r_{i_1}+\dots +r_{i_{l-1}}}\Z_p&
p^{r_{i_2}+\dots +r_{i_{l-1}}}\Z_p &
p^{r_{i_3}+\dots +r_{i_{l-1}}}\Z_p &\dots  &\GL_{d'-i_l}(\Z_p), 
\end{array}\end{displaymath}
where $*$ indicates an arbitrary matrix.

By identifying $\Lamb$ with a coset representative $\beta G_\nu$,
we can think of the columns of
$\beta$ as points in $\P^{d'-1}(\Z_p)$ and then list them as
$\beta_{{i_j},k}$ 
where $i_j\in I$ indicates the block that  $\beta_{i_j,k}$ belongs
to,
and $k\in\{i_{j-1}+1,\dots,i_j\}$ 
is the running index across the columns of the matrix. 
Moreover, by the action of the stabiliser, we can multiply any
column $\beta_{i_j,k}$ by a unit, add multiples of $\beta_{i_{j_1},k_1}$
to $\beta_{i_{j_2},k_2}$, whenever $k_2>k_1$ (and necessarily
$i_{j_2}\geq i_{j_1}$), and also 
add $p^{r_{i_{j_1}}+\dots + r_{i_{j_2}-1}}\beta_{i_{j_2},k_2}$ to 
$\beta_{i_{j_1},k_1}$ when $j_2>j_1.$
If 
$b_{i_{j_1},n}^{i_{j_2},m}$ denotes the $(n,m)$-entry of $\beta$, the
above operations imply 
that $b_{i_{j_1},n}^{i_{j_2},m}\in \Z_p/(p^{r_{i_{j_1}}+\dots+ r_{i_{j_2}-1}}).$
This observation enables us to compute the number of lattices of a fixed elementary divisor type.

\begin{Def}
For each lattice $\Lamb$ of type $(I,r_I)$ we define the
\emph{multiplicity} of $\Lamb$, which we denote by $\mu(\Lamb)$, to be
the number of lattices 
of fixed type $(I,r_I)$, divided  by the factor $b_I(p)$ which counts the number of $\F_p$-points on
the corresponding
flag variety.\end{Def}

We now define an expression which encodes this multiplicity as a
function of $(I,r_I)$. One 
can check that in order to compute $\mu(\Lamb)$ we may assume that $b_{i_{j_1},n}^{i_{j_2},m}\in p\Z_p/(p^{r_{i_{j_1}}+\dots+ r_{i_{j_2}-1}}).$
The
function $\mu:\N\times \N\ra \N$ will measure the size of the set of $x\in
p\Z_p/(p^{a})$ of a fixed $p$-adic valuation as follows: 

\begin{Def}
Let $a,b$ be fixed positive integers. We define a binary function 
$\mu:\N\times \N\ra \N$ as follows \begin{equation*} \mu(a,b):=|\left\{x\in
p\Z_p/(p^{a}):v_p(x)=b\right\}|=\begin{cases} 1 & \text{if
$a=b$}\\ p^{a-b}(1-p^{-1}) & \text{if $a>b$}\\ 0
&\text{otherwise.}\end{cases}
\end{equation*}\end{Def}
This definition extends naturally to a $(n+1)$-ary function $\mu:\N\times \N^n\ra \N$; if
$\mathbf{b}=(b_1,\dots,b_n)$ is a vector then
$$\mu(a;{\bf
b}):=|\left\{\mathbf{x}\in (p\Z_p/(p^{a}))^n:v_p(x_i)=b_i\right\}|.$$ We note that
$\mu(a;b_1,b_2,b_3,\dots,b_n)=\mu(a,b_1)\mu(a,b_2)\dots\mu(a,b_n)$ and
\begin{equation}\label{sum mu}\sum_{b=1}^{a}\mu(a,b) =
p^{a-1}.\end{equation}

The next result records the multiplicity of a lattice of given type
$\nu=(I,r_I)$.

\begin{prop}\label{multiplicity}
Let $\Lamb$ be a maximal lattice of type $\nu=(I,r_I)$ with a coset 
representative
$\beta G_\nu$ under the 1-1 correspondence in (\ref{corr}), where $I=\{i_1,\dots,i_l,i_{l+1}\}$ and $\beta\in\GL_{d'}(\Z_p)$. Write $b_{i_{j_1},k}^{i_{j_2},m}$ for the $(k,m)$ entry of $\beta$, where
$k\in\{i_{j_1-1}+1,\dots,i_{j_1}\}$ and $m\in\{i_{j_2-1}+1,\dots,i_{j_2}\},$ so that the pair $(i_{j_1}, i_{j_2})$ indicates the block of this entry.
Then 
\begin{align*}\mu(\Lamb) &=
\prod_{\underset{j_1<j_2}{j_1,j_2\in
\{1,\dots,l,l+1\}}}
\prod_k\prod_m
\sum_{b_{i_{j_1},k}^{i_{j_2},m}=1}^{r_{i_{j_1}}+\dots+ r_{i_{j_2-1}} }\mu (
r_{i_{j_1}}+\dots+ r_{i_{j_2-1}},b_{i_{j_1},k}^{i_{j_2},m})\\
& = \prod_{\underset{j_1<j_2}{j_1,j_2\in
\{1,\dots,l,l+1\}}}
p^{i_{j_1}i_{j_2}(r_{i_{j_1}}+\dots+ r_{i_{j_2-1}}-1)}
= p^{-\dim\flag_I +\sum_{i_j\in I} (d'-i_j)i_jr_{i_j}},\end{align*} where $d'=i_{l+1}$ and $I=\{i_1,\dots,i_l,i_{l+1}\}$.
\end{prop}

\begin{proof}
By excluding the factor $b_I(p)$, we can assume that we are inside the
first congruence subgroup, 
see \cite{Abscissa} page 327 for details. Now the result follows by carefully enumerating the entries in the matrix $\beta$ and by
applying the equation (\ref{sum mu}) above.
\end{proof}

\section{Grassmannians and flag varieties}\label{flags}

The set of $(k-1)$-planes in $\P^{d'-1}= \P(V)$ admits a variety
structure via the standard 
Pl\"ucker embedding into $\P^N=\P(\bigwedge^k V)$, where $N=\binom{d'}{k}$.

We can give local affine coordinates for open subsets on the Grassmannian $G(k,d')$ in the following way. Let $\Gamma\subset V$ be a subspace of dimension $d'-k$
corresponding to a multivector $\omega\in \bigwedge^{d'-k} V$. Now $\omega$ can be thought of as a homogeneous linear form on $\P(\bigwedge^k V)$, and we define $U\subset\P(\bigwedge^kV)$ to be the affine open subset defined by $\omega\neq 0$. Thus $U\cap G(k,V)$ is just the set of $k$-dimensional subspaces $\Delta\subset V$ that are complementary to $\Gamma$.

To give this construction in coordinates, let $\Gamma\subset V$ be a subspace of dimension $d'-k$, spanned by $e_{k+1},\dots, e_{d'}\in \F_p^{d'}$, say. Then $U\cap G(k,V)$ is the subset of spaces $\Delta$ such that the $k\times d'$ matrix $M_\Delta$ whose first $k\times k$ minor is non-zero. It follows that any $\Delta\in U\cap G(k,V)$ can be represented as the row space of a unique matrix of the form 
\begin{equation}\label{affine}
\begin{array}({ccccccc cc })
1&0&0&\dots &0& a_{11}&a_{12}&\dots&a_{1,d'-k}\\
0&1&0&\dots &0& a_{21}&a_{22}&\dots&a_{2,d'-k}\\
0&0&1&\dots &0& a_{31}&a_{32}&\dots&a_{3,d'-k}\\
\vdots &&&&&&&&\\
0&0&0&\dots &1& a_{k1}&a_{k2}&\dots&a_{k,d'-k}
\end{array}\end{equation} and vice versa. The entries of this matrix then give the bijection of $U\cap G(k,V)$ with $\F_p^{k(d'-k)}$.

A flag variety $\flag_I$ is a generalisation of a Grassmannian.  Here
$I=\{i_1,\dots,i_l\}_<\subseteq\{1,\dots,d'\}$ and $\flag_I$
is a subvariety of a product of Grassmannians defined by the incidence correspondence 
$$\flag_I:=\{(\Pi_1,\dots,\Pi_l):\Pi_1\subset\dots\subset \Pi_l\} \subset \G(i_1-1,d'-1)\times\dots\times \G(i_l-1,d'-1)\}.$$
By generalising the above construction  for affine coordinates of points on the Grassmannian, we can choose local coordinates for points on the flag variety. 

We can also define flags of type $I=\{i_1,\dots,i_l\}$ starting from the lattices $\Lamb$ of flag type $I$ in the following way.
As before,
let $\Lamb$ be a lattice of type $\nu=(I,r_I)$ which corresponds to the
coset $\beta G_\nu$ under the 1-1 correspondence (\ref{corr}), for some fixed $\beta\in
\GL_{d'}(\Z_p)$. Write $\beta_{i_j,k}$ for the columns of $\beta$.
Let $\overline{\beta_{i_j,k}}$ denote the reduction of $\beta_{i_j,k}$
mod $p$. These $\overline{\beta_{i_j,k}}$ can be thought of as points in $\P^{d'-1}(\F_p)$.
We define vector spaces for each $i_j\in I$ by setting
$$V_{i_j}=\<\overline{\beta_{i_1,1}},\dots,\overline{\beta_{i_j,i_j}}
\>_{\F_p}<\F_p^{d'}.$$ We observe that
$\mathrm{dim}_{\F_p}(V_{i_j})=i_j.$ We call 
$(V_{i_j})_{i_j\in I}$ the flag of type $I$ associated to $\Lamb$.
Obviously, a number of lattices will give the same flag mod $p$.

\begin{Def}\label{lift}
Let $\mathfrak{F}\in\flag_I$, where
$I=\{i_1,\dots,i_l\}_<$, and let $\Lamb$ be a maximal lattice of type $I$. Then $\Lamb$ is said to be a lift of $\mathfrak{F}$ if its associated flag
$(V_{i})_{i\in I}$ is equal to $\mathfrak{F}$.
\end{Def}

The stabiliser of each flag $\mathfrak{F}\in\flag_I$ is the standard parabolic subgroup $P_I$ in $\GL_{d'}(\F_p)$. Each of these parabolic subgroups contain the standard Borel subgroup of upper triangular matrices.

\begin{Def}\label{equivalent vectors}
Two maximal lattices $\Lamb_1$ and $\Lamb_2$ of type $(I,r_I)$ are said to be  \emph{equivalent} if they are lifts of the same flag $\mathfrak{F}\in\flag_I$. 
\end{Def}

This is defines an equivalence relation $\sim$ on the set of maximal lattices. We note that $\Lamb_1\sim \Lamb_2$ if and only if the respective coset representatives of  $\Lamb_1$ and $\Lamb_2$ differ by an action of right multiplication by the parabolic subgroup $P_I$ in $\GL_{d'}(\Z_p)$. 
In particular, this observation allows us to change the coordinates of the flag, as well as providing a uniform way to move between different cosets corresponding to liftings of the same flag.

Let $\Lamb$ be  a maximal lattice of type $I=\{i_1,\dots,i_l\}_<,$ corresponding to  $\beta$ with columns
$\{\beta_{i_1,1},\dots,\beta_{i_1,i_1},\beta_{i_2,i_1+1},\dots,\beta_{i_2,i_2},\dots
\beta_{i_l,i_l},\dots \beta_{d',d'}\}$, and  entries $b_{i_{j_1},k}^{i_{j_2},m}$. We will construct a lattice $\Lamb'$ such that $\Lamb \sim \Lamb'.$ To do this, first set
\begin{equation}\label{borel}B=
\begin{array}({ccccc})
\lambda_1 &\lambda_1 & \lambda_1 & \dots &\lambda_{1}\\
 &\lambda_2 & \lambda_2& \dots &\lambda_{2}\\
& & \lambda_3 & \dots &\lambda_{3}\\
&&&\ddots&\vdots\\
 & &  & &\lambda_{d'}\\
\end{array},\end{equation} where all the entries are $p$-adic units. This is clearly an element of the standard Borel subgroup, so $B$ is contained in each of the parabolic subgroups $P_I$, and thus multiplying $\beta G_\nu$ from the right by $B$ leaves the associated flag invariant. We have now moved to the coset  $\beta G_\nu B =\alpha G_\nu$,  and as before we shall
denote the columns of $\alpha$
by 
$\{\alpha_{i_1,1},\dots,\alpha_{i_1,i_1},\alpha_{i_2,i_1+1},\dots,\alpha_{i_2,i_2},\dots
\alpha_{i_l,i_l},\dots, \alpha_{d',d'}\},$ which expand as
\begin{equation}\label{affpts}
\alpha_{i_{j},k}=
\left(
    \begin{array}{c}
a_{i_{j},k}^{i_{l+1},d'}\\
\vdots\\
a_{i_{j},k}^{i_{j_2},m}\\
\vdots\\
a_{i_{j},k}^{i_{1},1}
    \end{array} \right) =
\lambda_1\left(
    \begin{array}{c}
b_{i_{1},1}^{i_{l+1},d'}\\
\vdots\\
b_{i_{1},1}^{i_{j_2},m}\\
\vdots\\
b_{i_{1},1}^{i_{1},1}
    \end{array} \right)+\dots+
\lambda_n\left(
    \begin{array}{c}
b_{i_{j_1},n}^{i_{l+1},d'}\\
\vdots\\
b_{i_{j_1},n}^{i_{j_2},m}\\
\vdots\\
b_{i_{j_1},n}^{i_{1},1}
    \end{array} \right)+ \dots
          +\lambda_k\left(
    \begin{array}{c}
b_{i_{j_1},k}^{i_{l+1},d'}\\
\vdots\\
b_{i_{j_1},k}^{i_{j_2},m}\\
\vdots\\
b_{i_{j_1},k}^{i_{1},1}
    \end{array} \right).
\end{equation}

However, this is not enough, since even if the reductions mod $p$ span the same flags, we will also require that
the coordinates of the $\alpha_{i_j,k}$ lie in the same quotient ring as the coordinates of  $\beta_{i_j,k}$.
This is not the case here, since e.g. $a_{i_{j},k}^{i_{l+1},d'}\in p\Z_p/(p^{r_{i_1}+\dots+r_{i_l}}),$
while $b_{i_{j},k}^{i_{l+1},d'}\in p\Z_p/(p^{r_{i_{j}}+\dots+r_{i_l}})$.
In order to make the coordinates of the $\alpha_{i_j,k}$ to lie in the
same quotient ring as those of $\beta_{i_j,k}$ we recall that
the $\alpha_{i_j,k}$ are unique up to the action of right-multiplication by $G_\nu$, and this will give us a reduction of the form
$b_{i_{j_1},n}^{i_{j_2},m} \mod p^{r_{i_{j_1}}+\dots+r_{i_{j-1}}}$,
so we may indeed assume each of the coordinates of the $\alpha_{i_j,k}$ lie in the same quotient of $p\Z_p$ as the coordinates of  $\beta_{i_j,k}$. 
To avoid unnecessary notation, all the coordinates where this kind of reduction has been performed will be
denoted by  $\hat{a}_{i_{j_1},k}^{i_{j_2},m}$.

\section{Solution sets for systems of linear congruences}\label{congruences}

Let $\Lamb$ be a maximal lattice of type $I=\{i_1,\dots ,i_l\}_<$ corresponding to the coset
$\beta G_\nu$. Let
$\{\beta_{i_1,1},\dots,\beta_{i_1,i_1},\beta_{i_2,i_1+1},\dots,\beta_{i_2,i_2},\dots
\beta_{i_l,i_l}\}$ denote the columns of $\beta$.
We want to determine the index of the kernel of the following system of linear
congruences for each $i_j\in I$ and  $k\in\{i_{j-1}+1,\dots,i_j\}$
\begin{equation}
\label{conditions}
\g M(\beta_{i_j,k})\equiv 0 \mod p^{r_{i_j}+\dots+r_{i_l}},
\end{equation}
where $\g=({g}_1,\dots,{g}_d)\in \Z_p^d$. And $M(\mathbf{y})$ is the matrix of relations in a presentation of $\Gamma$ as in Definition \ref{pfaffian}.

We shall denote the index of the kernel of the above system of linear congruences by $w'(\Lamb)$ and call it the weight function of the lattice $\Lamb$.
This index has a group theoretical significance, which we shall explain in Section \ref{basic theory}. 

The main point of this section is determine the value of this function for
different kinds of lattices $\Lamb$.

It is clear from (\ref{conditions}) that the solution set of
this system of linear congruences depends on
whether the matrix $M(\beta_{i_j,k})$ has full rank or not, i.e.,
depending on whether the determinant of the matrix $M(\mathbf{y})$ evaluated at $\beta_{i_j,k}$
is zero mod $p$ or not. Recall, that we insist that the Pfaffian is not identically zero. The vanishing locus of the square root of the determinant of
$M(\mathbf{y})$ defines a hypersurface in $\P^{d'-1}(\F_p)$. This hypersurface is called the Pfaffian hypersurface,
as in Definition \ref{pfaffian}, and we will denote it by
$\mathfrak{P}_G$.

Let us briefly sketch the connection with the arithmetic geometry which arises in this context and the type of geometric problems we encounter.
For convenience, let us consider a lattice of type $(p^{r_{i_1}},\dots,p^{r_{i_1}},1,\dots,1)$ which gives rise to column vectors
$\{\beta_{1},\dots,\beta_{i_1},\beta_{i_1+1},\dots,\beta_{d'}\}$  so that $\beta_{i_1+1},\dots,\beta_{d'}$ give redundant congruence conditions and there exists $k\in\{1,\dots,i_1\}$ such that $\det M(\beta_k)\equiv 0\mod p$ with  respect to the system of linear congruences. 

The $\beta_1,\dots,\beta_{i_1}$ are unique only up to multiplication by a unit and addition of $\Z_p$-multiple of $\beta_i$ to $\beta_j$ for any $i,j\in\{1,\dots,i_1\}$, so it follows that either all mod $p$ reductions $\overline{\beta_k}$  of $\beta_k\in\P^{d'-1}(\Z_p)$ are  $\overline{\beta_1},\dots,\overline{\beta_{i_1}}\in \mathfrak{P}_G$ or all $\overline{\beta_1},\dots,\overline{\beta_{i_1}}\not\in \mathfrak{P}_G$. These cases correspond to whether the $(i_1-1)$-plane $\langle\overline{\beta_1},\dots,\overline{\beta_{i_1}}\rangle$ is contained on $\mathfrak{P}_G$ or not. 

As mentioned in Section \ref{flags} the $(i_1-1)$-planes in $\P^{d'-1}$ are parametrised by a Grassmannian variety. The $(i_1-1)$-planes that are contained on $\mathfrak{P}_G$ are parametrised by the Fano variety $F_{i_1-1}(\mathfrak{P}_G)$.
This Fano variety may be reducible, and we will denote its components by $F_{\i_{i_1}}$.

The most convenient description of this Fano variety is as a subvariety of the Grassmannian.
Let us now describe a construction of its defining ideals on the Grassmannian. Choose an $(i_1-1)$-plane on the Pfaffian.
Using the construction of the affine coordinates given in Section \ref{flags}, we can identify an open neighbourhood of this point on the Grassmannian $\G(i_1-1,d'-1)$ as an affine space. In this neighbourhood, a hyperplane can be represented as the row span of the matrix (\ref{affine}) as in Section \ref{flags}, with coefficients $\mu_i$, say. Restricting this hyperplane onto the Pfaffian will give us a bihomogeneous polynomial of bidegree $(\frac{d}{2},\frac{d}{2})$. The coefficients of the monomials in the $\mu_i$ will give us the defining ideals of the Fano variety on the Grassmannian. We shall denote the dimension of the Fano variety by $d_{\i_{i_1}}$ and its codimension on the Grassmannian by $c_{{\i_{i_1}}}$.

It is however, not enough to know that the determinant of the matrix of relations is zero mod $p$ at the smooth points. We know the exact rank of the matrix or relations from the following lemma.  

\begin{lemma}[Voll \cite{singular}]\label{rank}
If $\mathfrak{P}_G\subseteq \P^{d'-1}(\Q)$ is smooth, then the rank of the matrix of relations $G$ at any smooth point is equal to $d-2$. 
\end{lemma}
In
particular, the matrix of relations at a smooth point always contains
a $(d-2) \times (d-2)$ minor with a $p$-adic unit determinant.

\begin{Def}
Let $A\in\mathrm{Mat}_{d,l}(\Z_p)$. Then the corank of $A$ is defined to be  $r=d-n$, where
$n\times n$ is the largest $p$-adic unit minor. 
\end{Def}

In view of, (\ref{conditions}) we are interested only in the
index of the kernel of the set of equations $$\g M(\beta_{i_j,k})\equiv 0\mod p^{r_{i_j}+\dots+r_{i_l}}$$
for $i_j\in I=\{i_1,\dots,i_l\}$ and $k\in\{i_{j-1}+1,\dots,i_j\}$.

\begin{remark}
It is important to note that only the index matters; this is the
single fact that makes calculating normal zeta functions of class
two nilpotent groups so much easier than calculating the
subgroup zeta function. \end{remark}

The information we need from the matrices $M(\beta_{i_j,k})$ for the computation of the index is independent of the basis chosen.
In the next few pages we show that, by a suitable choice of homogeneous coordinates, we can arrange the points on the Pfaffian so that
the solution sets of these congruences form a chain.

Let us first make a reduction.
\begin{lemma}\label{reducing weights}
Let $\Lamb$ be a lattice of flag type $I=J_1\cup J_2$ where $J_1=\{i_1,\dots,i_l \}$ and
$J_2=\{i_{l+1},\dots,i_m\}.$ Suppose that the intersection of the flag $(V_{i_j})_{i_j\in I}$ associated with $\Lamb$ with the Pfaffian hypersurface is the flag $\mathfrak{F}\in \flag_{J_1}$. Then
\begin{equation*}w'(\Lamb)=w'(\Lamb,I,r_I)=w'(\Lamb,J_1,r_{J_1})
+w'(\Lamb,J_2,r_{J_2}),
\end{equation*} where $J_1=\{i_1,\dots,i_l \}$ and
$J_2=\{i_{l+1},\dots,i_m\}.$ Moreover,
$$w'(\Lamb, J_2,r_{J_2})=d(r_{i_{l+1}}+\dots + r_{i_{m}}).$$ 
\end{lemma}

\begin{proof}
Let
$\{\beta_{i_1,1},\dots,\beta_{i_1,i_1},\beta_{i_2,i_1+1},\dots,\beta_{i_2,i_2},\dots
\beta_{i_m,i_m}\}$ be
the columns of $\beta$ in $\beta G_\nu$ corresponding to $\Lamb$. Then
$M(\beta_{i_1,1}),\dots, M(\beta_{i_1,i_1})$,$M(\beta_{i_2,i_1+1}),\dots,
M(\beta_{i_l,i_l})$ have determinants zero mod $p$, and we can assume that the determinants of
$M(\beta_{i_{l+1},i_l+1}),\dots,M(\beta_{i_m,i_m})$ are $p$-adic units. As before, we can do this since the $\beta_{i_j,k}$ are unique only up to the action of the stabiliser.

Writing this explicitly, we
need to solve the following systems of linear congruences
\begin{align*}
\g M(\beta_{i_j,k})&\equiv 0 \mod p^{r_{i_j}+\dots+r_{i_m}} &\forall i_j&\in J_1, k\in\{i_{j-1}+1,\dots,i_j \},\\
 \g &\equiv 0 \mod p^{r_{i_j}+\dots+r_{i_m}}&\forall i_j&\in J_2.
\end{align*}
which is equivalent to the two independent sets of equations
\begin{align*}
\g M(\beta_{i_j,k})&\equiv 0 \mod p^{r_{i_j}+\dots+r_{i_l}}&\forall i_j&\in J_1, k\in\{i_{j-1}+1,\dots,i_j\}, \\ 
\g &\equiv 0 \mod p^{r_{i_{l+1}}+\dots+r_{i_m}}.&&
\end{align*}
The first of these is equivalent to the congruence conditions that arise for
$\Lamb$ of type $(J_1,r_{J_1})$, while the  second corresponds to the congruences conditions for
$w'(\Lamb,J_2,r_{J_2})$. The latter is also equal to 
$$w'(\Lamb,J_2,r_{J_2})=d(r_{i_{l+1}}+\dots+r_{i_m})$$ as claimed.
\end{proof}

Let us now consider the possible intersections of a flag $\mathfrak{F}\in\flag_{J_1}$ of type $J_1=\{i_1,\dots,i_l\}$
associated to $\Lamb$ with the Pfaffian hypersurface of $G$.

\subsection*{Case I: The intersection is empty}

In the light of the previous lemma, first suppose that is $J_1=\emptyset$. Then
the weight function is $$w'(\Lamb)=d\left(\sum_{i\in J_2} r_i\right)=
d\left(\sum_{i\in I} r_i\right),$$ since $I= J_2$.

\subsection*{Case II: The intersection is a hyperplane}

Next let us assume the intersection is a hyperplane.

If $J_1=\{1\}$ then the desired result follows from the following:

\begin{lemma}[Voll \cite{FE}]\label{smthpt}
Let $\xi\in \mathfrak{P}_G$, and let $\Lamb$ be a lattice of
type $J_1=\{1\}$ such that the flag associated with $\Lamb$ intersects the Pfaffian exactly at this fixed point. Then 
is $w'(\Lamb,1,r_1)= dr_1-2\min\{r_1,v_p(\det M(\beta_{1,1}))\}$.
\end{lemma}

We include his proof here, since this is the template we shall use later on for a more complicated case analysis.
\begin{proof}
Choose $\beta_1\in \P^{d'-1}(\Z_p)$ such that $\det M(\beta_{1,1})
\equiv 0 \mod p,$ and $\overline{\beta_{1,1}} =\xi\in\mathfrak{P}_G.$
Let $J= \begin{array}({ c c  })0 &1\\ -1&0
\end{array}.$ Then we can choose local coordinates such that
around any of the $\n_1(p)$ points of the Pfaffian hypersurface
mod $p$ the congruence conditions look like
\begin{equation*}
\g\left( \begin{array}{ c c }0 &\det M(\beta_{1.1})\\ -\det
M(\beta_{1.1})&0 \end{array}, J, \dots,J\right) \equiv 0 {\mod
p^{r_1}}.\\
\end{equation*}
It follows that $w'(\Lamb,1,r_1)=
dr_1-2\min\{r_1,v_p(\det M(\beta_1))\}$, as claimed.
\end{proof}

Let $\Pi$ be a $(i_1-1)$-plane on the Pfaffian, and $\Lamb$ a lattice of type $J_1=\{ i_1\}$, for $i_1>1$, such that $\Pi$ is the intersection of the Pfaffian with the flag associated to $\Lamb$, i.e.  
$\<\ol{\beta_{i_1,1}},\dots,\ol{\beta_{i_1,i_1}}\>=\Pi$

As in Section \ref{flags}, we can consider the columns of $\alpha$ where 
$\alpha G_\nu=\beta G_\nu B$ instead $\beta$ and so we are required to compute the index
of the kernel  of
\begin{equation*}
  \g M(\alpha_{i_1,k})\equiv 0 \mod p^{r_{i_1}},
\end{equation*}
for $k=1,\dots,i_1$.
We can conveniently write this as the augmented matrix
\begin{equation}\label{block}
\g
\left(M(\alpha_{i_1,1})|M(\alpha_{i_1,2})|\dots|M(\alpha_{i_1,i_1})\right)\equiv
0\mod p^{r_{i_1}}.
\end{equation}

We need to consider separately all the possible components of the Fano variety $F_{i_1-1}(\mathfrak{P}_G)$.
The first thing to consider is the rank of the system of
congruences, which is determined by the number of linearly independent column vectors of the augmented matrix.
By Lemma \ref{rank}, the rank of the augmented matrix is always $d$, $d-1$ or
$d-2$, so that corank is either 0,1 or 2, respectively.

In fact, the corank is 2 only if the intersection is a single point. The next lemma demonstrates that the corank of the augmented matrix is either 0 or 1 for lines and higher dimensional subspaces.

\begin{lemma}
Let $\langle\overline{\alpha_1},\dots,\overline{\alpha_{i_1}}\rangle\in\mathfrak{P}_G$, for $i_1\geq 2$, then set of congruences  
\begin{equation*}
\g
\left(M(\alpha_{i_1,1})|M(\alpha_{i_1,2})|\dots|M(\alpha_{i_1,i_1})\right)\equiv
0\mod p^{r_{i_1}}
\end{equation*}
has corank 1 or 0.
\end{lemma}

\begin{proof}
It is enough to prove the lemma for the case of lines, that is $i_1=2$, since the rank of the augmented matrix in (\ref{block}) is minimal in the case of lines for general $i_1\geq 2$.
From Lemma \ref{smthpt}
we know that around a smooth point the corank is 2. Choose homogeneous coordinates for this point (as a vector),
and choose coordinates for another smooth point on the Pfaffian such that these two points span a line on the
Pfaffian. These vectors are linearly independent and we can add multiples of one to the other and we can also
multiply them by units. Thus without loss of generality (up to a permutation of entries) our vectors look like
$\sigma=(\sigma_1,\dots,\sigma_{d-2},0,1)$ and $\tau=(\tau_1,\dots,\tau_{d-2},1,0)$, where $\sigma_i,\tau_i\in p\Z_p/(p^{r_{2}})$.
We may assume that we are not in the
diagonal case. Indeed, if the matrix of
relations is a permutation of a diagonal matrix, then $G$ is a direct product product of Heisenberg groups, in which case the
Pfaffian is not smooth, so we
can exclude this case. Then the corank of $M(\sigma)$ is 2 and we only need to show that at
least one column of $M(\tau)$ is linearly
independent from the columns of $M(\sigma)$.  But this is indeed the case, since the matrix of relations
is anti-symmetric, so the unit entry in $\tau$
will necessarily be in a different row to any unit entry in $M(\sigma)$. Thus the corank is at most 1,
and if we happen to have two more linearly
independent vectors then the corank is 0.
\end{proof}

First let us assume the corank is 0. Here the augmented matrix contains a
$p$-adic unit minor of size $d\times d$ , and so the desired result follows, since
the congruences will reduce to $$\g\equiv 0 \mod p^{r_{i_1}}$$ as in Case I: The intersection is empty, hence
$w'(\Lamb,i_1,r_{i_1})=w'(i_1,r_{i_1})=dr_{i_1}$.

\begin{remark}\label{corank 0} It is for these components $F_{\i_i}(\mathfrak{P}_G)$ for which we will set $\delta_{\i_i}=0$, as in statement of Theorem \ref{main theorem general}.
\end{remark}

Finally, let us assume we are in the corank 1 case.
As in the proof of Lemma \ref{smthpt},
we can use row and column operations to get each of the $M(\alpha_{i_1,k})$ for $k\in\{1,\dots,i_1\}$ into the form
$$\left(
\begin{array}{ c c }0 &\det M(\alpha_{i_1,k})\\ -\det M(\alpha_{i_1,k})&0
\end{array}, J,\dots,J\right),$$ where  $J= \begin{array}({ c c  })0 &1\\ -1&0
\end{array}.$

In order to determine
the index of $\g$ we need to know the rank, the determinant and the elementary
divisor type of the matrices $M(\alpha_{i_j,k})$, and these do not depend on the particular basis chosen. So it suffices to solve 
\begin{align*}
  \g \left( \begin{array}{ c c }0 &\det
M(\alpha_{i_1,1})\\ -\det M(\alpha_{i_1,1})&0 \end{array},
J,\dots,J\right)&\equiv 0 \mod p^{r_{i_1}}\\
  \g \left( \begin{array}{ c c }0 &\det
M(\alpha_{i_1,2})\\ -\det M(\alpha_{i_1,2})&0 \end{array},
J,\dots,J\right)&\equiv 0 \mod p^{r_{i_1}}\\ \vdots & \\ \g \left(
\begin{array}{ c c }0 &\det M(\alpha_{i_1,i_1})\\ -\det
M(\alpha_{i_1,i_1})&0 \end{array}, J,\dots,J\right)&\equiv 0 \mod
p^{r_{i_1}}
\end{align*} simultaneously, and we read off
\begin{equation}\label{pure weight 2}
w'(\Lamb,i_1,r_{i_1})=dr_{i_1}-\min\{r_{i_1},v_p(\det(M(\alpha_{i_1,1}))),v_p(\det(M(\alpha_{i_1,2}))),\dots,
v_p(\det(M(\alpha_{i_1,i_1})))\}.\end{equation}
Since we are assuming the corank is 1, we do not have the coefficient two in the minimum-function that we had in Lemma \ref{smthpt}.

The different components of the Fano variety may give different weight functions. This is either due to the rank condition, or the possibility that the components have different dimensions.  We
shall denote the $\F_p$-points on the component $F_{\i_{i_1}}$ by $\n_{\i_{i_1}}$. If a component has corank 0, we get the same weight function as the empty intersection case, and then we set $\delta_{\i_{i_1}}=0$.

\subsection*{Case III: The general intersection}

Here we assume that the intersection of the flag of type $I=\{i_1,\dots,i_l\}_<$ with the Pfaffian is not just a hyperplane but a
flag with subspaces contained on the Pfaffian. For simplicity, let us exclude the corank 0 cases, since these again will give the same weight function as the case of an empty intersection.
As before, we shall consider the congruence conditions separately for different lattices lifting fixed flags on the Pfaffian. Moreover, as in the case of a hyperplane intersection we consider the points
$\alpha_{i_j,k}$ defined in Section \ref{flags} equation (\ref{affpts}), instead of the $\beta_{i_j,k}.$
In
order to solve the equations against the same modulus, we 
consider the set of equations
\begin{align*}
p^{r_{i_1}+\dots+r_{i_{j-1}}}\g  M(\alpha_{i_j,k})\equiv 0 \mod
p^{r_{i_1}+\dots+r_{i_{j-1}}+r_{i_j}+\dots+r_{i_l}},
\end{align*}
where $i_j\in J_1$ and $k\in\{i_{j-1}+ 1,\dots,i_j\}$.

For each fixed $i_j$, we get an augmented matrix as in the hyperplane case:
\begin{align*}
p^{r_{i_1}+\dots+r_{i_{j-1}}}\g(
M(\alpha_{i_j,i_{j-1}+1})|\dots|M(\alpha_{i_j,i_j}))\equiv 0 \mod
p^{r_{i_1}+\dots+r_{i_{j-1}}+r_{i_j}+\dots+r_{i_l}}.
\end{align*}
With a suitable change of basis, each component of the augmented matrix reduces to
\begin{equation}\label{modulus}
p^{r_{i_1}+\dots+r_{i_{j-1}}}\g \left(\begin{array}({ c c}) 0&
\det M(\alpha_{i_j,k})\\ -\det M(\alpha_{i_j,k})
&0\end{array},J,\dots,J \right) \equiv 0 \mod
p^{r_{i_1}+\dots+r_{i_j}+\dots+r_{i_l}}.
\end{equation}
Moreover, as we vary $i_j\in J_1$, we get the whole system of
congruences consisting of augmented matrices  
of the same form as those which arose in the case $J_1=\{i_1\}$. As before, the elementary row and column operations do not change
the elementary divisor type of the matrices $M(\alpha_{i_j,k})$, so without loss of
generality we may assume all of the matrices are of the form above (\ref{modulus}). We can then read off the
weight function:
\begin{align*}
&w'(\Lamb,J_1,r_{J_1})=d(r_{i_1}+\dots+r_{i_l})\\&-\min\{r_{i_1}+\dots+r_{i_l},v_p(\det
M(\alpha_{i_1,1})),\dots,v_p(\det M(\alpha_{i_1,i_1})),\\ &r_{i_1}+v_p(\det
M(\alpha_{i_2,i_1+1})),\dots,r_{i_1}+v_p(\det M(\alpha_{i_2,i_2})),\dots,\\&
r_{i_1}+\dots+r_{i_{l-1}}+v_p(\det
M(\alpha_{i_l,i_{l-1}+1})),\dots,r_{i_1}+\dots+r_{i_{l-1}}+v_p(\det
M(\alpha_{i_l,i_l}))\}.
\end{align*}

However, if the flag contains the subspaces of points, i.e. if $J_1=\{1,i_2,\dots,i_l\}$, then
\begin{equation}\label{mixed weight}
\begin{split}
&w'(\Lamb, J_1, r_{J_1})=d(r_1+\dots+r_{i_l})-\min\{r_1,v_p(\det
M(\alpha_{1,1}))\}\\&
-\min\{r_1+\dots+r_{i_l},v_p(\det
M(\alpha_{1,1})),r_1+v_p(\det M(\alpha_{i_2,2})),\dots, r_1+v_p(\det M(\alpha_{i_2,i_2}))
,\\ &r_1+r_{i_2}+v_p(\det
M(\alpha_{i_3,i_2+1})),\dots,r_1+r_{i_2}+v_p(\det M(\alpha_{i_3,i_3})),\dots,\\&
r_1+r_{i_2}+\dots+r_{i_{l-1}}+v_p(\det
M(\alpha_{i_l,i_{l-1}+1})),\dots,r_1+r_{i_2}+\dots+r_{i_{l-1}}+v_p(\det
M(\alpha_{i_l,i_l}))\}.
\end{split}\end{equation}

\section{$p$-adic valuations of determinants}\label{valuations}

In order to get an explicit expression for the weight function $w'(\Lamb)$ we are left with the determination of $v_p(\det M(\alpha_{i_j,k})).$
This is easy from the following lemma:

\begin{lemma}\label{valuation}
Let $\alpha_{i_j,k}$ be the columns of $\alpha$ that we constructed in Section \ref{flags}, equation (\ref{affpts}), with the appropriate reductions.  
Then \begin{equation*}v_p(\det
M(\alpha_{i_j,k}))=\min\{v_p(a_1),\dots,v_p(a_{c_{\i_k}})  \},
\end{equation*} where $c_{\i_k}$ is the codimension of the component $F_{\i_k}$ of $F_{k-1}(\mathfrak{P}_G)$ in $\G(k-1,d'-1)$.
\end{lemma}

\begin{proof} 
First we note that we can expand  $\det
M(\alpha_{i_j,k})$ as a bihomogeneous polynomial of bidegree $(\frac{d}{2},\frac{d}{2})$ 
in the entries of $\alpha_{i_j,k},$ and in the $\lambda_n$, where $\lambda_n$ are the $p$-adic unit entries of the matrix $B$ (see equation (\ref{borel})) 
which we used in order to move between lattices lifting the same flag.
Writing this as a polynomial in the $\lambda_n$, the coefficients which appear are the defining ideals of the Fano variety on the Grassmannian.
Furthermore, the $p$-adic valuation of any monomial $v_p(\lambda_1^{\varps_1}\cdots\lambda_n^{\varps_n})=1,$
where $\sum_{i=1}^n \varps_i=\frac{d}{2}$. Hence
$$v_p(\det M(\alpha_{i_j,k}))\geq
\min\{v_p(a_1),\dots,v_p(a_{c_{\i_k}})  \},$$ 
where $a_i$ are polynomials in the entries of $\alpha$.
To see that equality holds,
suppose for a contradiction that the left hand side is strictly bigger
than the right hand side, so we have the
congruence
$$\sum_{\varps_1+\dots+\varps_n=\frac{d}{2}}\lambda_{1}^{\varps_1}\dots\lambda_{n}^{\varps_n}
a_{\varps_1,\dots,\varps_n} \equiv 0\mod p^{\kappa+1},$$ where $\kappa=\min\{v_p(a_1),\dots,v_p(a_{c_{\i_k}})  \}$.
Then
$$p^{\kappa}(\sum_{\varps_1+\dots+\varps_n=\frac{d}{2}}\lambda_{1}^{\varps_1}\dots\lambda_{n}^{\varps_n} a'_{\varps_1,\dots,\varps_n}) \equiv 0 \mod
p^{\kappa+1},$$ where $p^\kappa a'_{\varps_1,\dots,\varps_n}=a_{\varps_1,\dots,\varps_n}$ and thus $$\sum_{\varps_1+\dots+\varps_n=\frac{d}{2}}\lambda_{1}^{\varps_1}\dots\lambda_{n}^{\varps_n} a'_{\varps_1,\dots,\varps_n}
 \equiv 0
{\mod p}$$ for all possible $p$-adic units $\lambda_n$.
It follows that 
$$a'_{\varps_1,\dots,\varps_n}\equiv 0 {\mod p}$$ for all possible choices of $\{\varps_i\}$ with $\sum_{i=1}^n \varps_i=\frac{d}{2}$.
But this is a contradiction, since at least one of the $a'_{\varps_1,\dots,\varps_n}$ is a
$p$-adic unit.
\end{proof}

We have set up the new coordinates such that
the set of valuations  $v_p(\det M(\alpha_{i_j,k_1}))$ is
contained in the set of valuations $v_p(\det M(\alpha_{i_j,k_2}))$ whenever
$k_1<k_2$.
By applying Lemma \ref{valuation}, we can expand
the determinants in the weight function.

We start with the case $I=\{i_1\}$ so the intersection of the flag variety with the Pfaffian hypersurface is a hyperplane. In view of, (\ref{pure weight 2}) we need to compute
\begin{equation}\label{valuaatio}\min\{r_{i_1},v_p(\det(M(\alpha_{i_1,1}))),v_p(\det(M(\alpha_{i_1,2}))),\dots,
v_p(\det(M(\alpha_{i_1,i_1})))\}.\end{equation} From the definition of the $\alpha_{i_j,k}$,
we get
\begin{equation*}
v_p(\det(M(\alpha_{i_1,k})))=\min\{v_p(a_1),v_p(a_2),\dots,v_p(a_{c_{\i_k}})\}
\end{equation*} for each $k\in\{1,\dots,i_1\}$, where the $a_n$ are polynomials in the coordinates
of the points $\alpha_{i_1,k}$, and hence in the coordinates of the
$\beta_{i_1,k}$. Substituting in (\ref{valuaatio}) and cancelling the
extra minima and any terms that appear more than twice, we obtain
\begin{equation*}
w'(\Lamb,i_1,r_{i_1})=dr_{i_1}-\min\{r_{i_1},v_p(a_1),\dots,v_p(a_{c_{\i_{i_1}}})\}.
\end{equation*}

Finally, we consider the case of a general lattice with weight function as in
(\ref{mixed weight}). Here we need to expand
\begin{align}&\min\{r_{i_1}+\dots+r_{i_l},v_p(\det
M(\alpha_{i_1,1})),\dots,v_p(\det M(\alpha_{i_1,i_1})),\nonumber\\ &r_{i_1}+v_p(\det
M(\alpha_{i_2,i_1+1})),\dots,r_{i_1}+v_p(\det M(\alpha_{i_2,i_2})),\label{bigmin}\\ &
\dots,r_{i_1}+\dots+r_{i_{l-1}}+v_p(\det
M(\alpha_{i_l,i_{l-1}+1})),\dots,r_{i_1}+\dots+r_{i_{l-1}}+v_p(\det
M(\alpha_{i_l,i_l}))\}\nonumber.\end{align}

Again, by expanding out the valuations of the determinants, we set
\begin{equation*}
v_p(\det(M(\alpha_{i_j,k})))=\min\{v_p(\hat{a}_1),v_p(\hat{a}_2),\dots,v_p(\hat{a}_{c_{\i_{i_{j-1}}}}),v_p(a_{c_{\i_{i_{j-1}+1}}}),\dots,v_p(a_{c_{\i_k}})\}
\end{equation*} for each $k\in\{i_{j-1}+1,\dots,i_j\},$
where $\hat{a}_n$ denotes the reduction
mod $p^{r_{i_1}+\dots+r_{i_{j-1}}}$ as explained in Section
\ref{flags}, in particular in connection with equation (\ref{affpts}).
Now we can substitute the above expansions back inside (\ref{bigmin}), and cancel the redundant minima inside the expression. We can also cancel all the sums with reductions in them, as the minimum cannot be attained at
these points
because
\begin{equation}\label{hat}r_{i_1}+\dots+r_{i_{j-1}}+v_p(\hat{a}_{n})\geq
v_p(a_{n})\end{equation} for any $n\in\{1,\dots,c_{\i_{i_{j-1}}}\}.$

Let 
$\mathbf{r_{i_1}+r_{i_2}+\dots+r_{i_{j-1}}+a_{c_{\i_{i_j}}}}$ denote the $(c_{\i_{i_j}}-c_{\i_{i_{j-1}})}$-tuple
$$(r_{i_1}+r_{i_2}+\dots+r_{i_{j-1}}+v_p(a_{c_{\i_{i_{j-1}+1}}}),\dots ,r_{i_1}+r_{i_2}+\dots+r_{i_{j-1}}+v_p(a_{c_{\i_{i_j}}})),$$
where $c_{\i_{i_j}}$ is the codimension of $F_{\i_{i_j}}$  in $\G(i_j-1,d'-1)$. Then our weight function becomes
\begin{align*}
&w'(\Lamb,J_1,r_{J_1})\\&=d(r_{i_1}+\dots+r_{i_l})\\&-\min\{r_{i_1}+\dots+r_{i_l},v_p(a_{c_{\i_{i_1}}}),
\mathbf{r_{i_1}+a_{c_{\i_{i_2}}}},\mathbf{r_{i_1}+r_{i_2}+a_{c_{\i_{i_3}}}},\dots,
\mathbf{r_{i_1}+\dots+r_{i_{l-1}}+a_{c_{\i_{i_l}}}} \}
\end{align*}
where the highest dimensional linear subspace contained in the
Pfaffian, or the dimension of the subspace in a flag where the corank of the augmented matrix of relations changes to $0$, is $i_l-1$. Recall Remark \ref{corank 0}.

By applying Lemma \ref{reducing weights}, we deduce that the general weight function for a lattice
$\Lamb$ of type $I=\{i_1, \dots,i_m\}_<$ is given by
\begin{align*}
&w'(\Lamb,I,r_I)\\&=d(r_{i_1}+\dots+r_{i_m})\\&-\min\{r_{i_1}+\dots+r_{i_l},v_p(a_{c_{\i_{i_1}}}),
\mathbf{r_{i_1}+a_{c_{\i_{i_2}}}},\mathbf{r_{i_1}+r_{i_2}+a_{c_{\i_{i_3}}}},\dots,\mathbf{r_{i_1}
+\dots+r_{i_{l-1}}+a_{c_{\i_{i_l}}}}
\}
\end{align*}
where the highest dimensional linear subspace contained in the
Pfaffian, or the dimension of the subspace in a flag where the corank 
of the augmented matrix of relations changes to $0$, is $i_l-1$. Recall Remark \ref{corank 0}.

If $1\in I$ then we have 
\begin{align*}
&w'(\Lamb,I,r_I)\\&=d(r_{i_1}+\dots+r_{i_m})-\min\{r_1,v_p(a_1)\}\\&-\min\{r_{i_1}+\dots+r_{i_l},v_p(a_{c_{\i_{i_1}}}),
\mathbf{r_{i_1}+a_{c_{\i_{i_2}}}},\mathbf{r_{i_1}+r_{i_2}+a_{c_{\i_{i_3}}}},\dots,\mathbf{r_{i_1}
+\dots+r_{i_{l-1}}+a_{c_{\i_{i_l}}}}
\}.
\end{align*}

\section{Decomposing the zeta function}\label{basic theory}

Now we can return to the main topic of the current paper.
Recall that, we are interested in the normal zeta function $$\zeta_G^\ns(s)=\sum_{H\ns_f G}|G:H|^{-s},$$ of a finitely generated, torsion-free,
class-two-nilpotent group  $G$. We shall now put the previous sections
back into this context.

The zeta function admits the following Euler product decomposition
$$\zeta_G^\ns(s)=\prod_{p\ prime} \zeta_{G,p}^\ns(s),$$ where $\zeta_{G,p}^\ns(s)$ counts subgroups of
finite $p$-power index only.

Let us again write $G=\Gamma\times Z^m$.
Let $\L$ be the corresponding Lie ring constructed as an image of $G$ under the $\log$-map using the Mal'cev correspondence.
Set $\L_p:=\L\otimes\Z_p$. Then for almost all primes $p$ we have from \cite{GSS} $$\zeta_{G,p}^\ns(s)=\zeta_{\L,p}^\ns(s)=
\zeta_{\L_p}^\ns(s).$$
Therefore, it suffices to count ideals in the associated Lie ring. 
Let $\L_p'$ denote the derived Lie ring.

\begin{lemma}\cite[Lemma 6.1]{GSS}
Suppose $\L_p/\L_p'\isom\Z_p^{d+m}$ and $\L_p'\isom \Z_p^{d'}$ as rings.
For each lattice $\Lambda'\leq
\L_{p}'$ put $X(\Lambda')/\Lambda'=Z(\L_{p}/\Lambda')$. Then
\begin{align*}
\zeta_{\L,p}^\ns(s)=
\zeta_{\L_p}^\ns(s)=&\zeta_{\Z_p^{d+m}}(s)\sum_{\Lambda'\leq
\L_p'} |\L_p':\Lambda'|^{d+m-s}|\L_p:X(\Lambda')|^{-s}\\ =&
\zeta_{\Z_p^{d+m}}(s)\zeta_p((d+d')s-d'(d+m))A(p,p^{-s}),
\end{align*}
where $$A(p,p^{-s})=\sum_{\underset{\Lamb
\mathrm{maximal}}{ \Lambda'\leq
\L_p'}}|\L_p':\Lambda'|^{d+m-s}|\L_p:X(\Lambda')|^{-s}.$$

\end{lemma}

This lemma allows us to restrict to lattices in the centre only. We
enumerate such lattices according to the elementary divisor
type of the lattice, as explained in Section \ref{correspondence}. It is enough to
consider only maximal lattices of $p$-power index, since  $\Lamb=p^{r_{d'}}\Lamb_{max}$ if $\Lamb$ is not maximal, where $\Lamb_{max}$ is maximal in its
class. Now $$|\L'_p:\Lambda'|=p^{d'r_{d'}}
|\L'_p:\Lambda'_{max}|$$ and
$$|\L_p:X(\Lambda')|=p^{dr_{d'}}|\L_p:X(\Lambda'_{max})|.$$ 

We define the weight functions
\begin{align*}w(\Lambda')&:=\log_p(|\L'_p:\Lambda'|)\\
w'(\Lambda')&:=\log_p(|\L_p:X(\Lamb)|),\end{align*} where $\Lamb$
is a maximal lattice. Then
\begin{equation*}
A(p,p^{-s })=\sum_{\underset{\Lamb
\text{maximal}}{ \Lambda'\leq
\L_p'}}
p^{(d+m)w(\Lamb)-s(w(\Lamb)+w'(\Lamb))}.
\end{equation*}

Note that the $w'(\Lamb)$ is precisely the function we considered in Section \ref{congruences}, while  $w(\Lamb)$ is easily read
off from the type of the lattice. Indeed, if the type of $\Lamb$ is $\nu=(I,r_I)$, then \begin{equation*}w(\Lamb)=\sum_{i\in I}ir_i.\end{equation*}

\section{Indexing}\label{indexing}

We can decompose the generating function $A(p,p^{-s})$ further to
run over lattices of fixed flag type, and write it as
\begin{equation}\label{generating general}
A(p,p^{-s})=\sum_{I\subseteq\{1,\dots,d'-1\}}A_I(p,p^{-s}),\end{equation}
where
\begin{equation}\label{actua
function general}A_I(p,p^{-s})=\sum_{\nu(\Lamb)=I}
p^{(d+m)w(\Lamb)-s(w(\Lamb)+w'(\Lamb))}\mu(\Lamb).\end{equation} Here
$\Lamb$ is a representative lattice of flag type $I$ and multiplicity $\mu(\Lamb)$, as in Proposition \ref{multiplicity}.

A
more subtle decomposition is needed in order to reveal the dependence on the underlying geometry. 

If the $(i_j-1)$-dimensional subspace of the flag of type $I=\{i_1,\dots,i_k\}_<$ is contained on the Pfaffian, we shall write $i_j^*$ in the indexing set
$I$ to indicate this fact. For example
$$A_{2^*}(p,p^{-s})=\sum_{\nu(\Lamb)={2^*}}
p^{(d+m)w(\Lamb)-s(w(\Lamb)+w'(\Lamb))}\mu(\Lamb)$$ where the sum is taken over lattices $\Lamb$ of flag type $\{2\}$ such that their associated flag
them consists only of a line which lies completely on the Pfaffian hypersurface.Thus with this notation we have
$I=\{i_1,\dots,i_k\}\in
\{1,\dots,d'-1,1^*,2^*,\dots,l^*\},$ where $l$ denotes the
highest dimensional linear subspace on the Pfaffian.

The indexing  and further decomposition of $A(p,p^{-s})$ are done via \emph{admissible} subsets $I\subseteq
\{1,\dots,{d'-1},1^*,2^*,\dots,l^*\}$. We call $I$  admissible if the following conditions hold:
\begin{enumerate}[(i)]
\item Only $i_j$ or $i_j^*$  can belong to $I$, but not both;
\item If $i_j^*\in I$ then $i_k\not\in I$ for all $k\leq j$.
\end{enumerate} 

Then we have
\begin{align*}
A(p,p^{-s})&=\sum_{I\subseteq\{1,\dots,d'-1\}}c_{I,p}A_I(p,p^{-s})+
\sum_{\underset{J_2\subseteq\{2,\dots,d'-1\}}{I=1^*\cup
J_2}}c_{I,p}A_I(p,p^{-s})+\\&
+\sum_{\underset{J_2\subseteq\{3,\dots,d'-1\}}{\underset{J_1\subseteq\{1^*\}}{I=J_1\cup
2^*\cup J_2}}}c_{I,p}A_I(p,p^{-s}) +\dots +
\sum_{\underset{J_2\subseteq\{l+1,\dots,d'-1\}}{\underset{J_1\subseteq\{1^*,\dots,l-1^*\}}
{I=J_1\cup l^*\cup J_2}}}c_{I,p}A_I(p,p^{-s}).
\end{align*}
It is possible to give an explicit description of the 
coefficients $c_{I,p}$.
First let $\n_{\i_{i_j}}(p)$ denote the number of $\F_p$-points on the 
component $F_{\i_i}$  of $F_{i-1}(\mathfrak{P}_G)$,
and let $\delta_{\i_i}$ be zero or one depending whether the corank of
this component is zero or one.
Also let $b_I(p)$  be the number of
$\F_p$-points on the flag variety defined by lattices of flag type
$I$.  This is equal to
$$b_I(p)=\binom{d'}{i_l}\binom{i_l}{i_{l-1}}\dots\binom{i_3}{i_2}
\binom{i_2}{i_1},$$ for $I=\{i_1,i_2,\dots, i_l\}\subseteq\{1,\dots,d'\}.$
We define the flag type 
$I-k:=\{i_j-k:i_j\in I\}\subseteq \{1,\dots,d'-k\}$. We claim that if
$I=\{i_1,\dots,i_n\}$ then \begin{equation}\label{sum1} c_{I,p}=
b_{I}(p)-b_{I- i_1}(p)(\sum_{\i_{i_1}}\delta_{\i_i}\n_{\i_{i_1}}(p)),
\end{equation} while for $I=\{i_1^*,\dots,i_n^*,k^*,j_1,\dots,j_r\}$.
\begin{equation}\label{sum2}c_{I,p}= b_{J_2-
k}(p)(\sum_{\i_{k}}\delta_{\i_i}\n_{\i_{k}}(p))b_{J_1}(p)-b_{J_2-(k-j_1)}(p)(\sum_{\i_{j_1}}
\delta_{\i_i}\n_{\i_{j_1}}(p))b_{J_1\cup
k}(p).
\end{equation}
This essentially comes from the formulae for the number of points on flag
varieties. In (\ref{sum1}) we have a flag
variety where no part of the flag intersects the Pfaffian
hypersurface, so at the level of $(i_1-1)$-spaces we need to subtract the
number $\sum_{\i_{i_1}}\delta_{\i_i}\n_{\i_{i_1}}(p)$ of $(i_1-1)$-spaces on the Pfaffian, for which the weight function is different.
The coefficient in (\ref{sum2}) is derived in a similar fashion. As
$(k-1)$-dimensional spaces lie on the Pfaffian we need to compute, $\sum_{\i_{k}}\delta_{\i_i}\n_{\i_{k}}(p)$. Restricting to any $(k-1)$-dimensional space of these it is clear that the flag variety with the highest subspace being this fixed $(k-1)$-dimensional spaces, lies completely on the Pfaffian.
In higher dimensions we count only those spaces that are off the Pfaffian at the level $j_1$,
and thus we subtract $\sum_{\i_{j_1}}\delta_{\i_i}\n_{\i_{j_1}}(p)$.

Rearranging such that the we can pull out the coefficients $\n_{\i_{i_j}}(p)$, we obtain the following decomposition
\begin{equation*}A(p,p^{-s})=W_0(p,p^{-s})+\sum_{i=1}^l\sum_{\i_i}\delta_{\i_i}
\n_{\i_i}(p)W_{\i_i}(p,p^{-s})\end{equation*} where
\begin{equation*}W_0(p,p^{-s})=\sum_{I\subseteq\{1,\dots,d'-1\}}
b_I(p)A_I(p,p^{-s})\end{equation*} and
\begin{align*}
W_{\i_i}(p,p^{-s})&=\sum_{\underset{J_2\subseteq\{i+1,\dots,d'-1\}}
{\underset{J_1\subseteq\{1^*,2^*,\dots,i-1^*\}}{I=J_1\cup i^*\cup
J_2}}}b_{J_2-i}(p)b_{J_1}(p)A_I(p,p^{-s})-
\sum_{\underset{J_2\subseteq\{i+1,\dots,d'-1\}}{\underset
{J_1\subseteq\{1^*,2^*,\dots,i-1^*\}}
{I=J_1\cup i\cup J_2}}}b_{J_2-i}(p)b_{J_1}(p)A_I(p,p^{-s})\\& =
\sum_{\underset{J_2\subseteq\{i+1,\dots,d'-1\}}
{J_1\subseteq\{1^*,2^*,\dots,i-1^*\}}}
b_{J_2-i}(p)b_{J_1}(p)(A_{J_1\cup i^* \cup J_2}(p,p^{-s}) -
A_{J_1\cup i\cup J_2}(p,p^{-s}) ).
\end{align*}
Here $i^*$ is a lattice that gives a point on $F_{\i_i}$, and $J_1$ runs over all lattices that are contained in this particular component $F_{\i_i}$.

\section{Igusa factors}\label{Igusa}

Now that we have an expression for the
generating function, and all the terms appearing in it, the rest of the proof of Theorem \ref{main theorem general} is basically a generalisation of the summation formulae that appear in \cite{F24}, see Lemmas 4.2 to 4.16, in particular.

\begin{lemma}\label{igusa1}
\begin{align*}W_0(p,p^{-s})&=\sum_{I\subseteq\{1,\dots,d'-1\}}
b_I(p)A_I(p,p^{-s})\\&=\sum_{I\subseteq\{1,\dots,d'-1\}}
b_I(p^{-1})\prod_{i\in I}\frac{X_i}{1-X_i}\\ &=
I_{d'-1}(X_1,\dots,X_{d'-1}),\end{align*} where
$X_i=p^{(d+d'+m-i)i-(d+i)s}.$
\end{lemma}

\begin{proof}
Using the Lemma \ref{multiplicity} we can write
\begin{align*}&\sum_{I\subseteq\{1,\dots,d'-1\}}
b_I(p)A_I(p,p^{-s})\\&=\sum_{I\subseteq\{1,\dots,d'-1\}}
b_I(p)\prod_{i\in I}\sum_{r_i=1}^\infty p^{(d+m)ir_i-(d+i)r_is}
p^{-\dim\flag_I}p^{(d'-i)ir_i}\\&=
\sum_{I\subseteq\{1,\dots,d'-1\}}b_I(p^{-1})\prod_{i\in I}\sum_{r_i=1}^\infty
p^{i(d+d'+m-i)r_i-(d+i)r_is}\\&=
\sum_{I\subseteq\{1,\dots,d'-1\}}b_I(p^{-1})\prod_{i\in I}
\frac{p^{i(d+d'+m-i)
-(d+i)s}}{1- 
p^{i(d+d'+m-i)-(d+i)s}}. 
\end{align*}
\end{proof}

\begin{lemma}\label{igusa2}
\begin{align*}
&W_{\i_i}(p,p^{-s})\\&=
\sum_{\underset{J_2\subseteq\{i+1,\dots,d'-1\}}{J_1\subseteq\{1^*,2^*,\dots,i-1^*\}}}
b_{J_2-i}(p)b_{J_1}(p)(A_{J_1\cup i^* \cup J_2}(p,p^{-s}) -
A_{J_1\cup i\cup J_2}(p,p^{-s}) ) \\ &= \sum_{J_2\subseteq
\{i+1,\dots,d'-1\}} b_{J_2-i}(p^{-1})\prod_{j_2\in
J_2}\frac{X_{j_2}}{1-X_{j_2}}
\sum_{J_1\subseteq\{1^*,2^*,\dots,i-1^*\}}b_{J_1}(p)(A_{J_1\cup
i^*}(p,p^{-s}) - A_{J_1\cup i}(p,p^{-s}) )\\ &=
I_{d'-i-1}(X_{i+1},\dots,X_{d'-1})\sum_{J_1\subseteq\{1^*,2^*,\dots,i-1^*\}}b_{J_1}(p)
(A_{J_1\cup
i^*}(p,p^{-s}) - A_{J_1\cup i}(p,p^{-s})),
\end{align*}
where $X_i=p^{(d+d'+m-i)i-(d+i)s}.$ Again $i^*$ is a lattice that gives 
a point on $F_{\i_i}$, and $J_1$ runs over all lattices that are contained on this particular component.
\end{lemma}

\begin{lemma}\label{igusa3}
\begin{align*}
&\sum_{J_1\subseteq\{1^*,2^*,\dots,i-1^*\}}b_{J_1}(p)(A_{J_1\cup
i^*}(p,p^{-s}) - A_{J_1\cup i}(p,p^{-s}) )\\ &=
\sum_{J_1\subseteq\{1^*,2^*,\dots,i-1^*\}}b_{J_1}(p^{-1})\prod_{{j_1}\in
J_1}\frac{Y_{\i_{j_1}}}{1-Y_{\i_{j_1}}}(A_{i^*}(p,p^{-s}) - A_{i}(p,p^{-s})
)\\&= I_{i-1}(Y_{1},\dots,Y_{\i_{i-1}})(A_{i^*}(p,p^{-s}) - A_{
i}(p,p^{-s}) ),
\end{align*}
where $Y_{\i_i}=p^{i(d+m)+c_{\i_i}-(d+i-1)s}$ for $\i_1>1$, while  
$Y_1=p^{d+m+c_1-(d-1)s}.$
\end{lemma}

It remains to calculate $A_{i^*}(p,p^{-s}) - A_{
i}(p,p^{-s}) $ and justify the terms $X_i$ and $Y_{\i_i}$ appearing in the
formulae above. 

\begin{prop}\label{basic sum} Let $G=\Gamma\times \Z^m$ as before. Let $d=h(G/Z(G))$, 
and $d'=h([G,G])$. Assume $d$ is even. Let
$n_i$ be the dimension of $\G(i-1,d'-1)$, $d_{\i_i}$ the
dimension of $F_{\i_i}$ and $c_{\i_i}$ its  codimension,  so  $n_i=c_{\i_i}+d_{\i_i}$. Then
\begin{align*}
A_{i^*}-A_i&=
\sum_{r_i=1}^{\infty}\sum_{a_1=1}^{r_i}\sum_{a_2=1}^{r_i}\dots
\sum_{a_{n_i}=1}^{r_i}\mu(r_i,a_1,a_2,\dots,a_{n_i})p^{i(d+m)r_i-{(d+i)r_i}s}
(p^{st_i\min\{r_i,a_1,a_2,\dots,a_{c_{\i_i}}\}}-1)\\
&=\frac{p^{i(d+m)-(d+i-t_i)s}(1-p^{-st_i})}{(1-p^{i(d+m)+d_{\i_i}-(d+i-t_i)s})
(1-p^{i(d+m)+n_i-({d+i})s})}=\frac{p^{-d_{\i_i}}Y_{\i_i}-p^{-n_i}X_i}{(1-Y_{\i_i})(1-X_i)},\end{align*}
where $X_i=p^{i(d+d'+m-i)-(d+i)s}$ and $Y_{\i_i}=p^{(i(d+m)+d_{\i_i})
-(d+i-t_i)s}$. Furthermore,
$t_1=2$ and $t_i=1$ for $i\geq 2$.
\end{prop}

\begin{proof}
The proof is given in Proposition 4.12 in \cite{F24} and is essentially a
manipulation of infinite series.
\end{proof}

We conclude that the $W_{\i_i}(p,p^{-s})$ are of the
form given in Theorem \ref{main theorem general}. This concludes the proof of Theorem \ref{main theorem general}.

\section{Example}\label{segre}

In this section show an explicit application of the main theorem, by
defining a class two Lie ring (recall, that there is also a class two
nilpotent group with the same presentation) whose Pfaffian is the quadric Segre surface in $\P^3$.

Let $G_\S$ have the presentation
$$G_\S=\<x_1,x_2,x_3,x_4,y_1,y_2,y_3,y_4:
[x_1,x_3]=y_1,[x_1,x_4]=y_2,[x_2,x_3]=y_3,[x_2,x_4]=y_4\>.$$
The matrix of relations
$M(\mathbf{y})_{ij}=[x_i,x_j],$ is 
\begin{equation*} M(\bf{y})=\begin{array}({c c c c}) 0 & 0 & y_1 &
y_2 \\ 0 & 0& y_3 & y_4 \\ -y_1& -y_3& 0&0\\ -y_2 & -y_4& 0& 0,\\
\end{array} \end{equation*} and the Pfaffian is thus defined
by $$\S: y_1y_4-y_2y_3=0.$$
Since $\S\isom \P^1\times\P^1$, the number of $\F_p$-rational points on $\S$ is
$|\S(\F_p)|=(p+1)^2$. Furthermore, over $\F_p$ there are
$2(p+1)$ lines on this surface and no higher dimensional linear subspaces, see e.g. \cite{Hirschfeld}.

Note that in this group $Z(G_\S)=[G_\S,G_\S]$, so in the application
of the Theorem \ref{main theorem general}, we have $m=0$.
\begin{theo}
For almost all primes $p$, the local normal zeta function of $G_\S$ is
given by
$$\zeta_{G_\S,p}^\ns(s)=\zeta_{\Z^4,p}(s)\cdot\zeta_p(8s-16)\cdot(W_0(p,p^{-s})+(p+1)^2W_1(p,p^{-s})+
2(p+1)W_2(p,p^{-s}))$$ where
\begin{align*}
W_0(p,T)&= I_3(X_1,X_2,X_3)\\ W_1(p,T)&=I_2(X_2,X_3)E_1(X_1,Y_1)\\
W_2(p,T)&=I_1(X_3)E_2(X_2,Y_2)I_1(Y_1)\end{align*} with
$X_i=p^{i(8-i)-(4+i)s}$, $Y_1=p^{6-3s}$ and $Y_2=p^{9-5s}.$
\end{theo}

\subsection{Calculation}

As before, the basic building blocks of the zeta function are the
generating functions over lattices of fixed elementary divisor types.
$$A_I(p,T)=\sum_{\nu(\Lambda')=I}
p^{dw(\Lambda')-s(w(\Lamb)+w'(\Lamb))}\mu(\Lamb).$$
The elementary divisor types of the lattice $\Lamb$ in this care are
$$(p^{r_1+r_2+r_3},p^{r_2+r_3},p^{r_3},1),$$
where $r_i\geq 0$, the flag type is $I:=\{i:
r_i>0\}\subseteq\{1,2,3\}.$ Immediately it follows that  
$w(\Lambda')={\sum_{i\in I}ir_i },$  and
the multiplicity $\mu(\Lamb)$ can be calculated using the
1-1 correspondence between lattices and cosets. 

\begin{ex}\label{mult}
Let $\Lamb$ have elementary divisors $(p^{r_1+r_2},p^{r_2},1,1)$ so
that type 
consists of $I=\{1,2\}$ and  $r_I=(r_1,r_2,0,0)$.
The stabiliser of $\Lamb$ 
takes the form
\begin{equation*}
\begin{array}({c c c})
\GL_1(\Z_p) & * & *\\
 p^{r_1}\Z_p&   \GL_1(\Z_p)& *\\
p^{r_1+r_2}\Z_p
&p^{r_2}\Z_p&\GL_2(\Z_p)
\end{array}.
\end{equation*}
If we then exclude the terms coming from the flag varieties, 
a coset of the stabiliser takes form
\begin{equation*}
\begin{array}({c c c c})
1&0&0 &0\\
b_3&1 &0 &0\\
b_2&a_2&1&0\\
b_1&a_1&0& 1
\end{array},
\end{equation*} where $a_1,a_2\in p\Z_p/(p^{r_2})$, 
$b_1,b_2\in p\Z_p/(p^{r_1+r_2})$, $b_3\in p\Z_p/(p^{r_1}).$
We need to calculate how many choices we can make for each $a_i,b_i$. 
Thus the 
multiplicity of $\Lamb$ is 
\begin{align*}
\mu(1,2)&=\sum
\mu(r_2;b_2)\mu(r_2,b_1)\mu(r_1+r_2;a_3)\mu(r_2+r_1;a_2)
\mu(r_1;a_1)\\&=p^{-5}p^{3r_1+4r_2}
\end{align*}
where the $a_i,b_i$ are variables, and the ranges of summations
are the obvious ones from equation (\ref{sum mu}).
\end{ex}

For $w'(\Lamb)$ we need to consider the flags associated with $\Lamb$,
defined in Section \ref{flags}.
Some
of the flags in $\F_p^4$ intersect with the Segre surface and in the
light of the 
congruence conditions,
the weight function $w'(\Lamb)$ changes. 

As in (\ref{conditions}) Section \ref{congruences}, in order to determine $w'(\Lamb)$ we want to 
understand the different solutions to the set of congruences
\begin{align*}\g M(\beta_1) &\equiv 0 \mod
p^{r_1+r_2+r_3}\\ \g M(\beta_2) &\equiv 0 \mod
p^{r_2+r_3}\\\g M(\beta_3) &\equiv 0 \mod
p^{r_3}.\end{align*} Note that we only insist that $r_i\geq 0$.

Since the Segre surface
does not contain any planes, the vector $\beta_3$ can always
be chosen 
such that $\det(M(\beta_3))$ is a $p$-adic unit and the third condition reduces to
\begin{equation*}\g \equiv 0 \mod p^{r_3},\end{equation*}
and we get an immediate reduction, as in Lemma \ref{reducing weights}, 
\begin{align}w'(1,3)&=w'(1)+w'(3)\notag\\
w'(2,3)&=w'(2)+w'(3)\notag
\\ w'(1,2,3)&=w'(1,2)+w'(3).\label{splitting weight}\end{align}
This induces a similar reduction in generating functions.
\begin{align}
A_{1^*,3}&=p^{2}\notag
A_{1^*}\cdot A_3\notag
\\ A_{1^*,2,3}&=p^{2}A_{1^*,2}\cdot A_3\notag\\
A_{2^*,3}&= p^{2}A_{2^*}\cdot A_3\notag \\ A_{1^*,2^*,3}&=p^{2}
A_{1^*,2^*}\cdot A_3,\label{reduce3}
\end{align}
because $w(J\cup 3)=w(3)+w(J)$, for
$J\subseteq\{1^*,2^*,1,2\}$ and the $\mu$-function has
enough multiplicative properties, e.g.,
$\mu(r_i+r_3;a_1)=p^{r_3}\mu(r_i;a_1)$. 

There are four additional
cases we need to consider
reflecting the four
different geometric configurations how a flag can intersect the Segre surface.
\begin{itemize}
\item [1.] $M(\beta_1)$ and $M(\beta_2)$ are both non-singular mod
$p$.
\item [2.] $M(\beta_1)$ is singular mod $p$, but $M(\beta_2)$ is not
  singular mod $p$, and the
line $\<\overline{\beta_1},\overline{\beta_2}\>$ does not lie on the Segre surface.
\item [3.] $M(\beta_1)$ and $M(\beta_2)$ are both singular mod $p$, and the line
$\<\overline{\beta_1},\overline{\beta_2}\>$ lies on the Segre surface. 
\item [4.] $M(\beta_1)$ and $M(\beta_2)$ are both singular mod $p$, and the flag
$\<\overline{\beta_1}\>< \<\overline{\beta_1},\overline{\beta_2}\>$ consisting of a point and
a line lies on the Segre surface.\end{itemize}

We are now justified to use the same indexing as in Section
\ref{indexing} and if we denote by $\n_1(p)$ the number of points on the Pfaffian,
and similarly by $\n_2(p)$ the number
of lines, the
generating function takes the explicit form
as a sum over all admissible subsets of $\{1^*,2^*,1,2,3\}$:
\begin{align*}A(p,p^{-s})&=A_\emptyset+ \left(\binom{4}{1}_p
-\n_1(p)\right)A_1+ \left(\binom{4}{2}_p-\n_2(p)\right)A_2\\ &
+\binom{4}{3}_pA_3+\binom{3}{1}_p\left(\binom{4}{1}_p
-\n_1(p)\right)A_{1,2}\\&
+\binom{3}{2}_p\left(\binom{4}{1}_p-\n_1(p)\right)A_{1,3}\\& +
\binom{2}{1}_p\left(\binom{4}{2}_p-\n_2(p)\right)A_{2,3}+
\binom{2}{1}_p\binom{3}{1}_p\left(\binom{4}{1}_p-\n_1(p)\right)A_{1,2,3}\\&
+\n_1(p)A_{1^*}+\left(\binom{3}{1}_p\n_1(p)-\n_2(p)\binom{2}{1}_p
\right)A_{1^*,2}+ \binom{3}{2}_p\n_1(p)A_{1^*,3}+\\&
\binom{2}{1}_p\left(\binom{3}{1}_p\n_1(p)-\n_2(p)\binom{2}{1}_p
\right)A_{1^*,2,3}+\n_2(p)A_{2^*}\\&
+\n_2(p)\binom{2}{1}_pA_{1^*,2^*}+\binom{2}{1}_p\n_2(p)A_{2^*,3}+
\binom{2}{1}_p\n_2(p)\binom{2}{1}_pA_{1^*,2^*,3}.
\end{align*}

We rearrange this sum in order to see exactly which parts
depend on $\n_1(p)$ and $\n_2(p)$, and obtain
\begin{equation}
A(p,p^{-s})=W_0(p,p^{-s})+\n_1(p)W_1(p,p^{-s})+\n_2(p)W_2(p,p^{-s}),
\end{equation}
where
\begin{align*}
W_0(p,p^{-s})&=\sum_{I\subseteq\{1,2,3\}} b_I(p) A_I(p,p^{-s}),\\
W_1(p,p^{-s})&=\sum_{I\subseteq\{2,3\}} b_I(p) (A_{1^*\cup
  I}(p,p^{-s})-A_{1\cup I}(p,p^{-s}))\\
W_2(p,p^{-s})&=\left(1+\binom{2}{1}_p p^2 A_3\right)
\left((A_{2^*}-A_2)+\binom{2}{1}_p\left(A_{1^*,2^*}-A_{1^*,2}\right)\right).
\end{align*} The third formula follows from the 
reductions in generating functions (\ref{reduce3}). By Lemma
\ref{igusa1} $W_0(p,p^{-s})=I_3(X_1,X_2,X_3)$, as wanted.


Now we need to determine the generating functions
$A_{1^*}-A_1$, $A_{2^*}-A_2$ and $A_{1^*,2^*}-A_{1^*,2}$, which will
finish this example.

The easiest case is  $A_{1^*}-A_1$. By Lemma \ref{smthpt} the lattices
$(p^{r_1},1,1,1)$ 
lifting a point on Pfaffian are in
1-1 correspondence with the vector $\beta_1=(1,a_1,a_2,a_3)^t,$ $a_i\in
p\Z_p/p^{r_1}$, and the weight function is
$w'(\Lamb)=4r_1-2\min\{r_1,v_p(\det M(\beta_1))\}$. By Lemma \ref{valuation} on
p-adic valuations, this is equal to 
$w'(\Lamb)=4r_1-2\min\{r_1,v_p(a_1)\}$
and hence we may apply
Proposition \ref{basic sum} and get 
\begin{align*}
A_{1^*}-A_1&=\frac{p^{4-3s}-p^{5-5s}}{(1-p^{5-3s})(1-p^{7-5s})}=E_1(X_1,Y_1),
\end{align*}
where $X_1=p^{7-5s}$ and $Y_1=p^{5-3s}$.

Next we calculate the weight function over lattices that lift lines 
on the Segre surface.
There are two rulings of
lines on the Segre surface, but in fact both of them behave
similarly. We consider the line defined by $y_1=y_2=0.$ Lattices $\Lambda'$ of
type $\{2^*\}$ 
lifting this line are in one-to-one correspondence with
pairs of vectors $\beta_1=(a_1,a_2,0,1)^t, \beta_2=(b_1,b_2,1,0)^t$ where
$a_i,b_i\in p\Z/(p^{r_2}).$ We will define an equivalent lattice as in Definition \ref{equivalent vectors}
\begin{equation*}\alpha_1=\begin{array}({c})\lambda_1a_1 \\ \lambda_1a_2\\0\\
\lambda_1\end{array},
\alpha_2=\begin{array}({c})\lambda_1a_1+\lambda_2b_1\\\lambda_1a_2+
\lambda_2b_2\\ \lambda_2\\ \lambda_1\end{array}.\end{equation*}
Now the congruences to be satisfied by $\g$ are
\begin{align*}
\g\begin{array}({c c c c}) 0 & 0 & \lambda_1a_1 & \lambda_1a_2 \\ 0 & 0& 0
& \lambda_1
\\ -\lambda_1a_1& 0& 0&0\\
-\lambda_1a_2  & -\lambda_1& 0& 0\\
\end{array}&\equiv 0\mod p^{r_2}\\
\g\begin{array}({c c c c}) 0 & 0 & \lambda_1a_1+\lambda_2b_1 & \lambda_1a_2+\lambda_2b_2 \\ 0 & 0& \lambda_2
&\lambda_1
\\ -\lambda_1a_1-\lambda_2b_1& -\lambda_2& 0&0\\
-\lambda_1a_2-\lambda_2b_2  & -\lambda_1& 0& 0\\
\end{array}&\equiv 0\mod p^{r_2}.
\end{align*}
The corank of this system of linear equations is 1, as
suitable row and column operations show. Thus we get the result as in
the equation  (\ref{pure weight 2}) the weight function is
$w'(\Lamb)=4r_2-\min\{r_2,v_p(\det(M(\alpha_1))),v_p(\det(M(\alpha_2)))\}.$ 
Now applying
the reduction Lemma \ref{valuations} in the $p$-adic valuations, 
we have to evaluate 
$\min\{r_2,v_p(\lambda_1^2a_1),v_p(\lambda_1^2a_1+
\lambda_1\lambda_2(b_1-a_2)-\lambda_1\lambda_2b_2)\}.$ Since
$\lambda_1,\lambda_2$ are $p$-adic units, this reduces to 
$\min\{r_2,v_p(a_1),v_p(b_1),v_p(b_2)\}.$ By
Proposition \ref{basic sum}
\begin{align*}
A_{2^*}-A_2&=\sum_{r_2=1}^\infty \sum_{a_i,b_i=1}^{r_2}
p^{8r_2}T^{6r_2-\min\{r_2,b_1,b_2,a_1\}}\mu(r_2,a_1,a_2,b_1,b_2)\\&=
\frac{p^{8-5s}(1-p^{9-6s})}{(1-p^{9-5s})(1-p^{12-6s})}=E_2(X_2,Y_2),
\end{align*}
for $X_2=p^{12-6s}$  and $Y_2=p^{9-5s}$.

Finally, we need to calculate the generating function
$A_{1^*,2^*}$ over a lattice of type $(p^{r_1+r_2},p^{r_2},1,1).$
A lattice $\Lambda'$ lifting a flag of type $\{1^*,2^*\}$, is in 
one-to-one correspondence with pairs of vectors
$$\beta_1=(a_1,a_2,a_3,1)^t$$ $$\beta_2=(b_1,b_2,1,0)^t$$ where $a_1,a_2\in
p\Z/(p^{r_1+r_2}),$  $a_3\in p\Z/(p^{r_1})$ and $b_1,b_2\in
p\Z/(p^{r_2}).$ Any other choice of a flag and its affine
neighbourhood behaves in the same way.
Again, change variables
\begin{equation*}\alpha_1=\begin{array}({c})\lambda_1a_1 \\ \lambda_1a_2\\\lambda_1a_3\\
\lambda_1\end{array},
\alpha_2=\begin{array}({c})\lambda_1\bar{a}_1+\lambda_2b_1\\\lambda_1\bar{a}_2+
\lambda_2b_2\\ \lambda_1\bar{a}_3+\lambda_2\\ \lambda_1\end{array},\end{equation*}
where $\bar{a}_i$ denotes the reduction mod $p^{r_2}$. As in the
previous case, the corank of the system of equations is 1, and thus
using the equation (\ref{mixed weight}) 
the weight function is
\begin{align*}w'(\Lamb)=&4r_2+4r_1-\min\{r_1,v_p(\det(M(\alpha_1)))\}-\\ &
\min\{r_1+r_2,v_p(\det(M(\alpha_1))),r_1+v_p(\det(M(\alpha_2)))\},\end{align*} and an application of Lemma \ref{valuations}  yields \begin{align*}
&\min\{r_1,v_p(\lambda_1^2(a_1-a_2a_3)) \}=\min\{r_1,v_p(a_1)
  \}\end{align*} and
\begin{align*}& \min\{r_1+r_2,v_p(\lambda_1^2(a_1-a_2a_3))
,r_1+v_p(\lambda_1^2(\bar{a}_1-\bar{a}_2\bar{a}_3)
+\lambda_1\lambda_2(b_1-\bar{a}_2-b_2\bar{a}_3)+\lambda_2^2(b_2))\}\\&
=\min\{r_1+r_2,v_p(a_1-a_2a_3)
,r_1+\min\{v_p(\bar{a}_1-\bar{a}_2\bar{a}_3),v_p(b_1-\bar{a}_2-b_2\bar{a}_3),v_p(b_2)\}\}\\&=
\min\{r_1+r_2,v_p(a_1-a_2a_3)
,\min\{r_1+v_p(\bar{a}_1-\bar{a}_2\bar{a}_3),r_1+v_p(b_1-\bar{a}_2-b_2\bar{a}_3),r_1+v_p(b_2)\}\}\\
&=\min\{r_1+r_2,v_p(a_1-a_2a_3)
,r_1+v_p(\bar{a}_1-\bar{a}_2\bar{a}_3),r_1+v_p(b_1-\bar{a}_2-b_2\bar{a}_3),r_1+v_p(b_2)\}\\
&=\min\{r_1+r_2,v_p(a_1)
,r_1+v_p(b_1),r_1+v_p(b_2)\}.\end{align*} Now 
we can write $A_{1^*.2^*}-A_{1^*,2}$ in a form where we can first
apply Lemma \ref{igusa3}
to get $$A_{1^*.2^*}-A_{1^*,2}=I_1(Y_1)(A_{2^*}-A_2)$$ 
and after that Proposition \ref{basic sum}, to obtain the final result
as above.

\subsection{A closed expression}

$$\zeta_{G_\S,p}^\ns(s)=\zeta_{\Z^4,p}(s)\zeta_p(8s-16)\zeta_p(5s-7)
\zeta_p(6s-12)\zeta_p(7s-15)\zeta_p(3s-6)\zeta_p(9s-5)\cdot
W(p,p^{-s}),$$ where
\begin{align*}W(p,T)&=1+p^4T^4+2p^5T^3-p^5T^5+2p^8T^5+p^9T^5-p^8T^6-p^9T^6+2p^{10}T^6\\
& +p^{11}T^7+p^{12}T^7+p^{13}T^7+
p^{14}T^7-p^{10}T^8-p^{11}T^8-p^{12}T^8+p^{13}T^8\\
&-2p^{13}T^9-2p^{14}T^9+2p^{15}T^9+p^{16}T^9+
p^{14}T^{10}-2p^{15}T^{10}-2p^{16}T^{10}\\&
+p^{17}T^{10}+2p^{18}T^{10}+2p^{19}T^{10}-2p^{17}T^{11}
-p^{18}T^{11}-2p^{19}T^{11}-p^{20}T^{11}\\&
-p^{18}T^{12}-2p^{19}T^{12} -p^{21}T^{12}
+p^{22}T^{12}+p^{23}T^{12}+p^{19}T^{13}
-p^{20}T^{13}\\&-p^{21}T^{13}-p^{22}T^{13}-p^{23}T^{13}+p^{24}T^{13}
+p^{20}T^{14}+p^{21}T^{14}-p^{22}T^{14}\\&-2p^{24}T^{14}-p^{25}T^{14}
-p^{23}T^{15}-2p^{24}T^{15}-p^{25}T^{15}-2p^{26}T^{15}
+2p^{24}T^{16}\\&+2p^{25}T^{16}+p^{26}T^{16}-2p^{27}T^{16}-2p^{28}T^{16}+p^{29}T^{16}
+p^{27}T^{17}+2p^{28}T^{17}\\&-2p^{29}T^{17}-2p^{30}T^{17}
+p^{30}T^{18}-p^{31}T^{18}-p^{32}T^{18}-p^{33}T^{18}
+p^{29}T^{19}\\&+p^{30}T^{19}+p^{31}T^{19}+p^{32}T^{20}+2p^{33}T^{20}-p^{34}T^{20}
-p^{35}T^{20}+p^{34}T^{21}\\&+2p^{35}T^{21}-p^{38}T^{21}
+2p^{38}T^{23}+p^{39}T^{23}+p^{43}T^{26}.\end{align*}

\noindent{\bf Acknowledgements} The author gratefully acknowledges
the support she has received by
Lady Davis and Golda Meir Postdoctoral Fellowships at the Hebrew
University of Jerusalem. The comments by Tim
Burness, Marcus du Sautoy and Christopher Voll were very helpful in
preparation of this paper.

\bibliographystyle{acm}
    \bibliography{bibJerusalem}

\end{document}